\documentclass[11pt,reqno]{amsart}
\usepackage{amsfonts,amssymb,latexsym}
\usepackage{amsmath,amsthm}
\usepackage{eucal}

\usepackage[pagebackref]{hyperref}  
\usepackage{color}

\usepackage[margin=3cm, a4paper]{geometry}

\newtheorem{theorem}{Theorem}[section]
\newtheorem{lemma}[theorem]{Lemma}
\newtheorem{proposition}[theorem]{Proposition}

\theoremstyle{definition}
\newtheorem{definition}[theorem]{Definition}
\newtheorem{hypothesis}{Hypothesis}

\theoremstyle{remark}
\newtheorem{remark}{Remark}[section]

\def\R{{\mathbb R}}
\def\K{{\mathbb K}}
\def\N{{\mathbb N}}

\def\ep{\varepsilon}
\newcommand{\T}{\mathbb{T}}
\newcommand{\Z}{\mathbb{Z}}
\def\supp{\mathop{\rm supp}\nolimits}
\newcommand{\F}{\mathcal{F}}
\newcommand{\ft}{\mathcal{F}}

\def\om{\omega}
\def\norm#1{\|#1\|}
\def\normo#1{\left\|#1\right\|}

\def\norm#1{\|#1\|}
\def\jb#1{\langle#1\rangle}

\def\wh#1{\widehat{#1}}

\def\jb#1{\langle#1\rangle}

\newcommand{\les}{{\lesssim}}

\newcommand{\Sch}{{\mathcal{S}}}

\newcommand{\cro}[1]{\langle #1 \rangle}

\numberwithin{equation}{section}

\newcommand{\EQ}[1]{\begin{equation}\begin{split} #1 \end{split}\end{equation}}
\newcommand{\EQN}[1]{\begin{equation*}\begin{split} #1 \end{split}\end{equation*}}

\begin{document}
\title[KP-I]{On the well-posedness of the KP-I equation }


\author{Zihua Guo and Luc Molinet}

\address{Zihua Guo, School of Mathematical Sciences, Monash University, Melbourne, VIC
3800, Australia}
\email{zihua.guo@monash.edu}
\address{Luc Molinet, Universit\'e de Tours, Universit\'e d'Orl\'eans, Parc Grandmont, 37200 Tours, France.}
\email{Luc.Molinet@lmpt.univ-tours.fr}

\date{\today}
\begin{abstract}
We revisit the local well-posedness for the KP-I equation. We obtain unconditional local well-posedness in $H^{s,0}(\R^2)$ for $s>3/4$ and unconditional global well-posedness in the energy space. We also prove the global existence of perturbations with finite energy of non decaying  smooth global solutions.
\end{abstract}
\maketitle
\setcounter{tocdepth}{1}

\tableofcontents

\section{Introduction}\label{section1}

In this paper, we study the Cauchy problem for the Kadomtsev–Petviashvili (KP) equations including KP-I
\begin{eqnarray}\label{eq:kpI}
\begin{cases}
( u_t+u_{xxx}+ u u_x )_x +\epsilon  u_{yy} =0;\\
u(x,y,0)=\phi(x,y),
\end{cases}
\end{eqnarray}
where $u(x,y,t):\R\times \R\times \R \rightarrow \R$ is the unknown function, 
$\phi$ is the given initial data and $\epsilon\in \{-1,1\}$. When $\epsilon = -1$, this equation is known as KP-I equation whereas when $ \epsilon=1$ it is known as the KP-II equation.  The purpose of this paper is  twofold : to prove some new results on the unconditional local wel-posedness (LWP) for KP equation \eqref{eq:kpI} and to give a simplified proof of the global well-posedness (GWP) for the KP-I equation in the energy space. We will also state some new results on the global existence of solutions to the KP-I equation with different behaviors at infinity. This will in particular take into account  perturbations of  (possibly non decaying) smooth global solutions to the KP-I equation. 

The KP equations arise in
physical contexts as models for the propagation of dispersive long
waves with weak transverse effects. They are generalizations to two spatial dimensions of the one-dimensional Korteweg–de Vries (KdV) equation. The KP-I equation models the case with strong surface tension (compared to gravitational forces) while the KP-II equation models the case with weak surface tension. The momentum and energy conservation lead to the following conservation quantities for solutions of the KP equations :
\begin{equation} \label{laws}
M(u) =  \int_{\R^2}\! u^{2}\quad \text{and} \quad 
E(u) =  \frac{1}{2} \int_{\R^2}\! u_x^2-\epsilon\frac{1}{2} \int_{\R^2}\! (\partial_x^{-1} u_y)^2-
\frac{1}{6}
\int_{\R^2}\!  u^3\; .
\end{equation}

The mathematical theory for KP equations was extensively studied since three decades but hereafter  we will focus on the well-posedness theory\footnote{Note that they  are known to be formally integrable and  some attempts to solve them by IST method can be found in \cite{W}, \cite{Zh}.}. Before quoting the different results, it is worth noticing that, due to the presence of $\epsilon $ in the quadratic part of the energy, only the conservation of the $ L^2$-norm may be useful to globalize the solutions of the KP-II equation whereas the energy (or even a higher order conservation law) may be  also useful to globalize the ones of the KP-I equation. 
For the KP-II equation, Bourgain \cite{Bourgain3} established global well-posedness for appropriate data in $L^2$ on both $\R^2$ and $\T^2$. Takaoka \cite{T} and then Takaoka and Tzvetkov \cite{TT} obtained local well-posedness in anisotropic Sobolev space $H^{s_1,s_2}$, which was further improved in \cite{Hadac} to the full subcritical cases, and later on by Hadac, Herr, and Koch \cite{HHK} to the sharp results in the critical space.

For the KP-I equation, according to \cite{MoSaTz1}, Picard iterative approaches for the KP-I equation fail in standard Sobolev space and $H^{s_1,s_2}$ for $s_1,s_2\in \R$ because the solution flow map is not $C^2$ smooth at the origin. In some weighted spaces, such as \cite{CIKS}, this technique remains effective. 
Local well-posedness in $H^{s,0}$ for $s>3/2$ was proved by Molinet, Saut and Tzvetkov \cite{MoSaTz4}.  Global well-posedness in the ``second" energy spaces was obtained in \cite{Kenig1, MoSaTz2}.  Later, Ionescu, Kenig and Tataru \cite{IKT}
proved global well-posedness in the natural energy space $\mathbb{E}^1=\{\phi\in
L^2(\R^2),\partial_x \phi \in L^2(\R^2), \partial_x^{-1}\partial_y \phi \in L^2(\R^2)\}$ by introducing the short-time $X^{s,b}$ space and energy estimate method.
The result of \cite{IKT} was improved and proof was simplified in \cite{Guo} where local well-posedness was proved in $H^{s,0}$ for $s\geq 1$.  For the partially periodic KP-I equation, GWP in the energy space was proved in \cite{Robert}.  

One purpose of this paper is to study the unconditional well-posedness in the anisotropic Sobolev space $H^{s_1,s_2}$
which is defined by
\[
H^{s_1,s_2}=\{\phi \in L^2(\R^2):
\norm{\phi}_{H^{s_1,s_2}}=\norm{\wh{\phi}(\xi,\mu)(1+|\xi|^2)^{s_1/2}(1+|\mu|^2)^{s_2/2}}_{L^2_{\xi,\mu}}<\infty\}.
\]
That is, to obtain uniqueness in $L^\infty_t H^{s_1,s_2}$, without other function spaces. The importance of this type of uniqueness result was  underlined by T. Kato \cite{Kato}. This  plays a crucial role in convergence analysis in the numerical study of PDE's and in the last decade many such  results have been obtained for  dispersive equations (see for instance \cite{GKO}, \cite{MosPil} and references therein).  
In the case of the KP-I equation, another purpose of this paper is to present a simplified proof of the global well-posedness in the energy space. To this aim we will also look at  solutions in the following spaces related to the energy :
$$
E^s(\R^2):=\{\phi \in H^{s,0}(\R^2)\;  \text{with} \; \partial_x^{-1} \phi_y \in H^{s-1,0}(\R^2) \} \; 
$$
equipped with its natural norm.
Now we state our main results:

\begin{theorem}\label{thmmain}\text{ }\\
(1) The KP equation \eqref{eq:kpI} is unconditionally locally well-posed in $H^{s,0}(\R^2)$ for $s>3/4$.\\
(2) The KP-I equation ($\varepsilon=-1$) is unconditionally globally well-posed in $E^s(\R^2) $ for $ s\ge 1$. 
\end{theorem}
Our strategy to prove the LWP result is similar to the one introduced in \cite{KoTz} to solve the Benjamin-Ono equation.  We attempt to establish a classical energy estimate on the solutions. For this,  we localize  in $x$-frequency each functions  of the trilinear term that appears in the energy estimate and we decompose the interval of time $[0,T] $ in small intervals of length $T/N $ where $ N $ is the highest $x$-frequency involved. On  each small interval of time, instead of applying the (linear) Strichartz estimates on the Duhamel formula as in \cite{KoTz}, we make use of the famous smoothing trilinear estimate in Bourgain's spaces proved in \cite{IKT}.  Up to our knowlegde, this is the first time  that this strategy is used directly on a trilinear estimate. The derivative loss in the Duhamel term is compensated in part by H\"older inequality in time on small interval and in part by the smoothing effect of the trilinear estimate. Re-summing on the intervals we obtain a trilinear estimate on $ [0,T] $ without any loss in the high space frequencies which is the main ingredient to prove local well-posedness. Note however that to handle the nonlinear term coming from the Duhamel part we actually also make use of the approach of \cite{KoTz}  based on Strichartz estimates  (see also \cite{Kenig1} or \cite{MoSaTz4} in the context of the KP-I equation) to get a priori bounds on some $ L^p_T L^\infty_{xy}$-norms of the solutions. To this aim we  extend the validity of the known Strichartz estimates  associated with the KP equation (see \cite{Saut}). As in \cite{IKT}, the difference of two solutions is then estimated in a Sobolev space with one less regularity in the first space variable and a weight on the  small associated frequencies. 
Finally note that  we were  inspired  by \cite{MoTa} where showing bilinear estimates by the same strategy,  is shown to be efficient to prove  unconditional LWP results for some  generalized KdV-like equation on the torus.\vspace*{2mm}
\begin{remark}\label{Rem31}
Our restriction $ s>3/4$ is linked to the  use of $L^\infty_T H^{s,0}$ as resolution space  that is necessary to establish the unconditional uniqueness. Indeed it can be checked that the result in \cite{Guo}, that used short time $ X^{s,b}$-spaces as resolution space,   can be  improved to get the LWP of the KP-I equation in $ H^{s,0}(\R^2) $ for $s>1/2$\footnote{During the submission process of this work this was written down  in a rigorous way  independently in \cite{Guo2} and  \cite{KSS}.} (see \eqref{nana} for a suitable estimate leading to an a priori estimate on the solutions in short time $ X^{s,b}$-spaces). However, to get an estimate in $L^\infty_T H^{s,0}$ we have to substitute the function
$ u_{N_3}$ in \eqref{nana} by its integral formulation on small intervals of time. The obstruction then comes from the contribution of the high-low interactions in the nonlinear term that is of the form 
$$
A:=\|\tilde{\eta} \int_{c}^t U_\pm (t-\tau) \partial_x P_{N_3} ( u_{\sim N_3} u_{\ll N_3} )(\tau) d\tau 
\|_{X^{0,\frac{1}{2},1}_\pm} 
$$
where $\tilde{\eta}$ is a smooth function that localized on a small interval $ I $ of lentgh $ |I| \sim N^{-1}$  with $ N\ge N_3 $ and $ c\in I$.
By \eqref{linearD2}, Bernstein inequality and Holder inequality in time we get 
$$
A\lesssim N^{-1/2} \|\tilde{\eta}_j\partial_x P_{N_3} ( u_{\sim N_3} u_{\ll N_3})\|_{L^2_{txy}}
\lesssim N_3 N^{-1/2} \|u_{\sim N_3}\|_{L^\infty_{I} L^2_{xy}}\|u_{\ll N_3}\|_{L^2_{I} L^\infty_{xy}}
$$
Applying directly this estimate, without taking into account the smallnes of $ |I| $, would lead to an a priori estimate in $ H^{1+,0}(\R^2) $ that is not suitable for us. 
On the other hand, by using Strichartz estimate in the small interval of time $I$ we are able to recover a  negative power of $ N $ in estimating $\|u_{\ll N_3}\|_{L^2_{I} L^\infty_{xy}}$, assuming that $ u $ belongs to $ L^\infty_T H^{\theta,0} $ for some $ \theta$ to specified. This is the aim of Lemma \ref{lemnew} at the end of Section \ref{Sec3}. However, since we do not control the $ L^\infty_{txy}$-norm of $ u $ we are  not be able to recover $N^{-1/2} $ that would be needed to get an a priori estimate in $H^{\frac{1}{2}+,0}(\R^2) $. Actually the best balance we reached is  Lemma \ref{lemnew} that ensures that 
$$
\|u_{\ll N_3}\|_{L^2_{I} L^\infty_{xy}} \lesssim N^{-1/4} C(\|u\|_{L^\infty_T H^{\frac{3}{4}+,0}})
$$
This leads to the a priori estimate in $H^{s,0}(\R^2) $ for $ s>3/4$. 
\end{remark}

In the last section we will look at solutions with different behavior at infinity. Indeed, it turns out that our approach may be easily adapted to such framework. The study of the existence of such solutions goes back to Zhidkov \cite{Zhidkov} in the KdV or NLS context. Improvements on the regularity assumptions on initial data for KdV-type equations can be found for instance in \cite{Ga}, \cite{P21} or \cite{Laurens}. In the context of the KP equations, results in this direction mainly concern some perturbations of  the line-soliton for which global existence is proven in \cite{MoSaTz5} for the KP-II equation and in \cite{MoSaTz4} for the KP-I equation. 

To tackle the study of possibly non decaying solutions, for a given bounded function $ \psi \;: \;[0,T] \times \R^2\to \R $,  we will look for solution of the KP equation of the form 
$ u=\psi+v $ where $v $ is a solution of the following Cauchy problem :
\begin{eqnarray}\label{KPnondecay}
\begin{cases}
( v_t+v_{xxx}+ v v_x +(v \psi)_x+g )_x  +\epsilon v_{yy} =0;\\
v(x,y,0)=\phi(x,y),
\end{cases}
\end{eqnarray}
where $ g=g(\psi) $ is a function decaying at infinity so that 
$$
g_x:=(\psi_t + \psi_{3x} + \psi \psi_x)_x +\epsilon \psi_{yy}\; . 
$$
Our first LWP result on \eqref{KPnondecay} is established under the following hypothesis on $ \psi $ (see the beginning of Section \ref{Sect2} for the definition of $ a+$ with $ a\in \R $ and of $J_x^\theta $ with $ \theta\in\R$).
\begin{hypothesis}\label{hyp1} 
$$
J_x^{s+1+} \psi  \in L^{\infty}_{loc} (L^\infty(\R^2)) \quad \text{and} \quad 
(\psi_t + \psi_{3x} + \psi \psi_x)_x +\epsilon \psi_{yy}=g_x \quad \text{with} \quad g  \in C(\R_+; H^{s,0}(\R^2)) \; .
$$
\end{hypothesis}
\begin{remark}
Note that this hypothesis is fulfilled  for instance when $\psi $ is a smooth (possibly non decaying) solution of the KP equation. For instance $\psi$ is a local solution of the KP equation in $ H^{s+2+}(\T^2) $,  $H^{s+2+}(\R\times \T)  $ or $H^{s+2+}(\T \times \R)$. One can also think of other type of solutions as the multi-line solitons of the KP-II equation that seem to be physically relevant (\cite{Kodama}, \cite{Kodama2}) or the genus N solutions (\cite{Krichever}). For a nice brief presentation of all these  exact solutions of the KP-II equation, we refer to  \cite{KS}. Finally we notice that we can also take any function $\psi =\Psi_x $ with $ \Psi \in C^1(\R_+; H^\infty(\R^2)) $.
\end{remark}
\begin{theorem}\label{thmmain2}
Let $ s>3/4 $ and $ \psi \, : \, \R^2\to \R $ satisfying Hypothesis \ref{hyp1}.
Then, for any $ v_0\in H^{s,0}(\R^2)$,  there exist $ T=T(\|v_0\|_{H^{\frac{3}{4}+,0}})>0 $  and a solution 
$v\in C([0,T]; H^{s,0}(\R^2)) $ 
 to \eqref{KPnondecay}. This solution is unique in the class $ L^\infty(0,T;H^{s,0}(\R^2)) $. \\
Moreover, for any $ R>0 $, the solution map
$$
{\mathcal G} \, :\, 
\begin{array}{rcl}
  B(0,R)_{H^{s,0}}  & \to & C([0,T(R)];H^{s,0}(\R^2)) \\
v_0 &\mapsto &v
\end{array} 
$$
  is continuous.
  \end{theorem}
  \begin{remark}
On account of the available LWP results on the KP equations this proves the LWP of the KP equations on $H^{s+2+}(\K) + H^{s,0}(\R^2) $, $s>3/4$,  for $ \K\in\{\T^2,\R\times\T, \T\times\R\} $.
\end{remark}

  For the KP-I equation, we also obtain a global existence result for finite energy perturbations of functions satisfying the following more restrictive hypothesis 
  \begin{hypothesis}\label{hyp2} 
$$
J_x^{s+1+} \psi  \in L^{\infty}_{loc} (L^\infty (\R^2)) , \; J_x^{(s-1)+} \psi_y   \in L^{\infty}_{loc} (L^\infty (\R^2)) 
$$
and  
$$
(\psi_t + \psi_{3x} + \psi \psi_x)_x - \psi_{yy}=g_x \quad \text{with} \quad  g   \in C(\R_+, E^{s}(\R^2)) \; .
$$
\end{hypothesis}

\begin{remark}
Note that again  this hypothesis is fulfilled    when $\psi $ is a smooth(possibly non decaying) global solution of the KP-I equation.  This is also the case when  $\psi =\Psi_{xx} $ with $ \Psi \in C^1(\R_+; H^\infty(\R^2)) $.
\end{remark}    

\begin{theorem}\label{thmmain3}
Let $ s\ge 1 $ and assume  that $ \psi \, : \, \R_+\times\R^2 \to \R $ satisfies 
Hypothesis \ref{hyp2}.
Then, for any $ v_0\in E^{s}(\R^2)$,  the solution to \eqref{KPnondecay}, with $ \varepsilon=-1$, constructed in Theorem \ref{thmmain2} exists for all positive times  and belongs to  
$ C(\R_+; E^{s}(\R^2)) $. 
Moreover, for any $ T>0 $, the solution map
$$
{\mathcal G} \, :\, 
\begin{array}{rcl}
  E^{s}  & \to & C([0,T];E^s(\R^2)) \\
v_0 &\mapsto &v
\end{array} 
$$
  is continuous.
\end{theorem}
\begin{remark}
This theorem leads to the global existence of finite energy perturbations of any global smooth solution to the KP-I equation on  $\R\times\T $, $ \T\times\T $ or $ \T\times\R $. As special cases, one can think of the line-soliton (KdV-soliton) : $u=\varphi_c(x-ct) $ where 
$$
\varphi_c(x,y)=\frac{3c}{2}\,{\cosh}^{-2}\Big(\frac{\sqrt{c}\, x}{2}\Big)
$$
or the Zaitsev  traveling wave (\cite{Z}) :$ u(t,x,y)=\psi_c(x-ct) $ where $\psi_c$
is localized in $x$ and periodic in $y$ :
\begin{equation}\label{1.3}
\psi_c(x,y)=2\alpha^{2}\kappa 
\frac{1-k^{-1/2}\cosh(\alpha x-ct)\cos(\delta y)}
{k(\cosh(\alpha x -c t)-\cos\delta y)^2},
\end{equation}
with
$$
\kappa=\frac{\delta^2}{\delta^2-\alpha^4},\quad
c=\alpha^3+\frac{3\delta^2}{\alpha},\quad \delta^2>\alpha^4,\quad
\alpha,\delta\in\R^{\star}\, .
$$
Note that previous results on the global existence of perturbations to the line-soliton  or Zaitsev traveling wave for the KP-I equation were restricted to perturbations at the level of the next conservation law (see \cite{MoSaTz4}) that verifies in particular the additional constraint $ \partial_{x}^{-2} v_{0,yy} \in L^2(\R^2)$.
\end{remark}    

\noindent

\section{Preliminary}\label{Sect2}

For $x, y\in \R$, $x\les y$ means that there exists $C>0$ such
that $x\leq Cy$. By $x\sim y$ we mean $x\les y$ and $y\les x$. 
For $p\in [1,\infty]$, we use $p'$ to denote the conjugate index, namely $\frac{1}{p}+\frac{1}{p'}=1$. 
For $a\in \R$, $a+$ denotes $a+\epsilon$ for any sufficiently small $\epsilon>0$, and similarly for $a-$.

For
$f\in \Sch'$ we denote by $\widehat{f}$ or $\ft (f)$ the Fourier
transform of $f$ for both spatial and time variables,
\begin{eqnarray*}
\widehat{f}(\xi,\mu,\tau)=\int_{\R^3}e^{-ix \xi}e^{-iy\mu}e^{-it
\tau}f(x,y,t)dxdydt.
\end{eqnarray*}
Moreover, we use $\ft_{x,y}$ and $\ft_t$ to denote the Fourier
transform with respect to space and time variable respectively. 

For $\theta\in\R $ we denote by $ J_x^\theta$ the Bessel potential of order $ -\theta$ with respect to the $x$-variable, i.e.
$$
J_x^\theta \varphi =\ft_{x,y}^{-1}((1+\xi^2)^{\theta/2}\ft_{x,y}(\varphi)) \;, \; \forall \varphi\in L^2(\R^2)\;.
$$

Fix a function $\eta: \R \rightarrow [0, 1]$ such that it is radial, smooth, 
supported in $\{|\xi|\leq 8/5\}$ and equals to $1$ for $|\xi|\leq 5/4$. Let $\phi(\xi)=\eta(|\xi|/2)-\eta(|\xi|)$. For $N\in 2^\Z$ let $\phi_N(\xi)=\phi(\xi/N)$, and define $P_N$ by
$\widehat{P_Nu}(\xi)=\phi_{N}(\xi)\widehat{u}(\xi,\mu)$. We then define in a standard way 
$$
P_{\le N} =\sum_{0<N_1\le N} P_{N_1} \quad \text{and} \quad P_{\ge N} =\sum_{N_1\ge N} P_{N_1}
$$
For $(\xi,\mu)\in \R \setminus \{0\}\times \R$ let
\begin{equation}\label{eq:dr}
\omega_\pm(\xi,\mu)=\xi^3\mp\mu^2/\xi
\end{equation}
and  write $\sigma_\pm (\tau,\xi,\mu)=\tau-\omega_\pm(\xi,\mu)$. Here $\omega_-=\xi^3+\mu^2/\xi$ is the dispersion for KP-I  whereas $\omega_+$ is the dispersion  for KP-II.

We define the Fourier multiplier $ Q_L^{\pm} $, $L\in 2^{\N} $, associated to the KP  group by 
\begin{align*}
\widehat{Q_L^\pm u}(\tau,\xi,\mu) =\phi_L( \sigma_\pm (\tau,\xi,\mu) ) \hat{u}(\tau,\xi,\mu) \quad \text{for } N\ge 2;\\
\quad \text{and}  \; \widehat{Q_1^\pm u}(\tau,\xi,\mu) =\Bigl(\sum_{0<L\le 1}\phi_L(\sigma_\pm (\tau,\xi,\mu) )\Bigr) \hat{u}(\tau,\xi,\mu).
\end{align*}
Let $U_\pm(t)\phi\in C(\R: L^2)$ denote the solution of the free KP  evolution given by
\begin{eqnarray*}
\ft_{x,y}[U_\pm(t)\phi](\xi,\mu,t)=e^{it\omega_\pm(\xi,\mu)}\widehat{\phi}(\xi,\mu).
\end{eqnarray*}
We will need to use  Bourgain spaces as intermediate spaces. Let us introduce 
$$\|u\|^2_{X^{0,1/2}_\pm}=
 \int _{\R_{\xi,\mu}^2\times\R_{\tau}}
 \langle \sigma_\pm (\tau,\xi,\mu )\rangle  |\hat{u}(\tau, \xi,\eta)|^2d\tau d\xi d\eta,
$$
and the following Besov version
\begin{equation}\label{deftX}
\|u\|_{X^{0,1/2,1}_\pm}=
\sum_{L\ge 1} \|Q_L u \|_{X^{0,1/2}} \; .
\end{equation}
In this paper, we  use the frequency envelope method (see for instance \cite{Tao04}  and \cite{KoTz}) in order to show the continuity result with respect to initial data.
To this aim, we first introduce the following:
\begin{definition}\label{deftenv}
  Let $\delta>1$.
  An {\it acceptable frequency weight} $\{\om_N^{(\delta)}\}_{N\in 2^\Z}$ is defined as a dyadic sequence satisfying $\om_N=1 $ for $ N< 1 $ and $  \om_N\le \om_{2N}\le \delta \om_N$ for $N\ge 1$.
  We simply write $\{\om_N\}$ when there is no confusion.
\end{definition}
With an acceptable frequency weight $\{\om_N\}$, we slightly modulate the classical Sobolev spaces in the following way:
for $s\ge0$, we define $H_\om^{s,0}(\R)$ with the norm
$$
  \|u\|_{H^{s,0}_\om}
  :=\bigg(\sum_{N\in 2^{\Z}}\om_N^2 \langle N\rangle^{2s}\|P_N u\|_{L^2}^2\bigg)^{\frac 12}.
$$
Note that $H_\om^{s,0}(\R^2)=H^{s,0}(\R^2)$ when we choose $\om_N\equiv 1$.

Next we collect some needed estimates. The main ingredients are the method developed by Koch and Tzvetkov (\cite{KoTz}) and the Strichartz estimates. 

\begin{definition} We say a pair $(q,r)$ is {\it admissible} if it satisfies the following conditions:
\begin{itemize}
    \item $2\leq q,r\leq \infty$;
    \item $\frac{1}{2}(\frac{1}{2}-\frac{1}{r})\leq \frac{1}{q}\leq \frac{1}{2}-\frac{1}{r}$;
    \item $(q,r)\notin \{(2,\infty), (4,\infty)\}$. 
\end{itemize}
\end{definition}

\begin{lemma} [Strichartz estimates]\label{Strichartz}
Let $T\in (0, \infty]$. Assume $(q,r), (\tilde q, \tilde r)$ are both admissible pairs. Then we have
\begin{equation}\label{hom}
\|D_x^{-\beta(q,r)} U_\pm(t)\phi  \|_{L^{q}_{T}L^{r}_{xy}}
\lesssim  \|\phi\|_{L^2},
\end{equation}
and
\begin{equation}\label{non-hom}
\left\|\int_{0}^{t}D_x^{-\beta(q,r)-\beta(\tilde q,\tilde r)}U_\pm(t-t')f(t')dt'\right\|_{L^{q}_{T}L^{r}_{xy}}
\lesssim \|f\|_{L^{{\tilde q'}}_{T}L^{{\tilde r'}}_{xy}},
\end{equation}
where $\beta(q,r)=3(\frac{1}{2}-\frac{1}{r}-\frac{1}{q})$.
\end{lemma}

\begin{proof}

The Strichartz estimates are known when assuming additionally $\beta(q,r) \le 1/2 $ (see \cite{Saut}, \cite{MoSaTz4}). Here we notice that this assumption is not needed. 

We only need to show \eqref{hom}.  The decoupled estimates in \eqref{non-hom} follows either from similar arguments or the Christ-Kiselev lemma since $q>2 $ and $ \tilde {q}'<2$. 

First we prove for any $N\in 2^\Z$, we have
\EQ{\label{eq:homproof1}
\|U_\pm(t)P_N\phi  \|_{L^{q}_{T}L^{r}_{xy}}
\lesssim  N^{\beta(q,r)}\|\phi\|_{L^2}.
}
By $TT^*$-method, \eqref{eq:homproof1} is equivalent to
\EQ{\label{eq:homproof2}
\normo{\int U_\pm(t-s)P_Nf ds }_{L^{q}_{T}L^{r}_{xy}}
\lesssim  N^{2\beta(q,r)}\|f\|_{L^{q'}_TL^{r'}_{x,y}}.
}
Note that 
$$
U_\pm(t) P_N\phi =G_\pm(\cdot,\cdot,t)\star \phi\, ,
$$
where $\star$ denotes the convolution with respect to the spatial variables
and $ G $ is defined by 
\EQ{
G_\pm(x,y,t):=&\int_{\R^2} e^{i(x\xi+y \mu)}\, e^{it(\xi^3\mp\mu^2/\xi)} \phi(\xi/N)\, d\xi
\, d\mu. 
}
Integrating $\mu$, we get
\EQ{
G_\pm(x,y,t)=&\int_{\R} t^{-1/2}\xi^{1/2} e^{i(t\xi^3+x\xi\mp {y^2\xi/t})} \phi(\xi/N)\, d\xi. 
}
Thus we get (see \cite{GPW}) for $\theta\in [0,1/2]$
\EQ{
|G_\pm(x,y,t)|\les |t|^{-1/2}
\sup_{x\in\R}\,\, \Bigl|\int_{\R} \xi^{\frac{1}{2}}\, e^{i(t\xi^3+x\xi)}\phi(\xi/N)d\xi\Bigr|\les |t|^{-\frac{1}{2}-\theta}N^{\frac{3}{2}-3\theta}.
}
Interpolating with the $L^2$ estimate, we get for $r\in [2,\infty]$ and $\theta\in [0,1/2]$.
\EQ{
\norm{U_\pm(t) P_N\phi}_{L^r_{x,y}}\les |t|^{(-\frac{1}{2}-\theta)({1-\frac{2}{r}})}N^{(\frac{3}{2}-3\theta)(1-\frac{2}{r})} \norm{\phi}_{L^{r'}_{x,y}}.
}
Thus by this and the Hardy-Sobolev inequality we get
\EQ{\label{eq:homproof3}
\normo{\int U_\pm(t-s)P_Nf ds }_{L^{q}_{T}L^{r}_{xy}}
\lesssim&  \normo{\int \norm{U_\pm(t-s)P_Nf}_{L^{r}_{xy}} ds }_{L^{q}_{T}}\\
\les& \normo{\int |t-s|^{(-\frac{1}{2}-\theta)({1-\frac{2}{r}})}N^{(\frac{3}{2}-3\theta)(1-\frac{2}{r})}\norm{f}_{L^{r'}_{xy}} ds }_{L^{q}_{T}}\\
\les& N^{(\frac{3}{2}-3\theta)(1-\frac{2}{r})}\|f\|_{L^{q'}_TL^{r'}_{x,y}}
}
provided that $0\leq (\frac{1}{2}+\theta)({1-\frac{2}{r}})<1$, $2<q\leq \infty$, $\frac{1}{q}=(\frac{1}{2}+\theta)({\frac{1}{2}-\frac{1}{r}})$. Rewriting $\theta$ in terms of $q,r$, we prove \eqref{eq:homproof2}.

Next we try to sum over $N$.  When $r<\infty$, using the Littlewood-Paley square function theorem we have
\EQN{
\norm{D_x^{-\beta(q,r)} U_\pm(t)\phi}_{L^{q}_{T}L^{r}_{xy}}\les& \norm{N^{-\beta(q,r)} U_\pm(t)P_N\phi}_{L^{q}_{T}L^{r}_{xy}l_N^2}\\
\les & \norm{N^{-\beta(q,r)} U_\pm(t)P_N\phi}_{l_N^2L^{q}_{T}L^{r}_{xy}}\\
\les &\|P_N\phi\|_{l_N^2 L^2}\les \|\phi\|_{L^2}.
}
When $r=\infty$, we couldn't use the square function theorem. Instead, we use $(\theta,1)$-real interpolation between the following two estimates
\EQN{
\norm{U_\pm(t)P_N\phi}_{l_2^{-\beta(q+,r)}L^{q+}_{T}L^{\infty}_{xy}}\les& \|\phi\|_{L^2}\\
\norm{U_\pm(t)P_N\phi}_{l_2^{-\beta(q-,r)}L^{q-}_{T}L^{\infty}_{xy}}\les& \|\phi\|_{L^2}.
}
By Theorem 5.6.2 and Theorem 5.2.1 in \cite{Inter}, we get
\EQ{
(l_2^{-\beta(q+,r)}L^{q+}_{T}L^{\infty}_{xy}, l_2^{-\beta(q-,r)}L^{q-}_{T}L^{\infty}_{xy})_{\theta,1}=l_1^{-\beta(q,r)}L^{q,1}_{T}L^{\infty}_{xy}.
}
Thus we can get 
\EQN{
\norm{D_x^{-\beta(q,r)} U_\pm(t)\phi}_{L^{q}_{T}L^{\infty}_{xy}}\les& \norm{U_\pm(t)P_N\phi}_{l_1^{-\beta(q,r)}L^{q}_{T}L^{\infty}_{xy}}\les \|\phi\|_{L^2}. 
}
The proof is complete.  

We note that for $r=\infty$, one can have an alternative proof using the frequency-global dispersive estimate: for $ 0\leq \varepsilon<3/2$
\begin{equation} \label{jhg}
\|D_x^{-\varepsilon}\, U_\pm(t)\phi\|_{L^\infty_{xy}}\lesssim   |t|^{-1+\ep/3}\norm{\phi}_{L^1} .
\end{equation}
Indeed, to show the above, 
integrating a gaussian integral with respect to $\mu$ (see \cite{Saut}), we get
for $ 0\leq \varepsilon<3/2$,
\begin{equation*} 
\|D_x^{-\varepsilon}\, G_\pm(x,y,t) \|_{L^\infty_{xy}}\lesssim   |t|^{-1/2}
\sup_{x\in\R}\,\, \Bigl|\int_{\R}|\xi|^{\frac{1}{2}-\ep}\, e^{i(t\xi^3+x\xi)}d\xi\Bigr| \les |t|^{-1+\ep/3}.
\end{equation*}
Then $(q,\infty)$-estimates follow from $TT^*$-method as before. 
\end{proof}

It follows in particular from this lemma that 
 for $ 4\le q < 8 $, 
\begin{equation}\label{est1}
\| U_\pm(t) \varphi \|_{L^q_{T}L^4_{xy}} \lesssim \|D_x^{\frac{3}{4}-\frac{3}{q} } \varphi\|_{L^2_{xy}}
\end{equation}
and for $ 2< q< 4 $ 
\begin{equation}\label{est2}
\| U_\pm(t) \varphi \|_{L^{q}_{T}L^\infty_{xy}} \lesssim \| D_x^{\frac{3}{2}-\frac{3}{q} } \varphi\|_{L^2_{xy}} .
\end{equation}
Next we collect some product estimates in the following lemma. 

\begin{lemma}[Product estimates]\label{LemProd} 
We have

(1) For $\theta\in\R $, $J_x^{|\theta|+} f\in L^\infty(\R^2) $ and $ g\in H^{\theta,0} $ it holds 
\begin{equation} \label{prod1}
\| f \, g \|_{H^{\theta,0}} \lesssim \|J_x^{|\theta|+} f\|_{L^\infty} \|g\|_{H^{\theta,0}}. 
\end{equation}

(2) Let $1\leq p,q\leq \infty$, $\frac{1}{2}\leq r\leq \infty $ satisfy $\frac{1}{p}+\frac{1}{q} = \frac{1}{r}$.  Assume $\theta>\max(0,\frac{1}{r}-1)$ or $\theta\in 2\N$.  Then it holds 
\begin{equation} \label{prod3}
\| J^{\theta}_x (f^2)\|_{L^{r}} \lesssim \| f\|_{L^p_{xy}} \|J^{\theta}_x f\|_{L^q_{xy}}.
\end{equation}

(3) For an acceptable frequency envelop $\{\omega_N\}_{N>0}$ we have
\begin{equation} \label{prod33}
\| f^2\|_{H^{s,0}_\omega} \lesssim \| f\|_{H^{s,0}_\omega} \| f\|_{L^\infty_{xy}}.
\end{equation}
\end{lemma}
\begin{proof}
For (2), see Corollary 1.1 in \cite{OhWu} and references therein. (1) and (3) follows from the standard Bony's paraproduct decomposition.  
\end{proof}

\section{Refined Strichartz estimates for solutions}\label{Sec3}
In this section we prove different a priori estimates on the $ L^2_T L^\infty_{xy} $-norm of solutions to the KP-I equation.
\begin{proposition}\label{pro1}
 Let $ 0<T<2 $ and $ u \in L^\infty_T H^{s,0}$ with $ s>2/3 $ be a solution to 
$$
(u_t + u_{xxx} +f_x)_x- u_{yy} =0
$$
then 
\begin{eqnarray}
\|  J_x^{s-(\frac{1}{2}+)} u \|_{L^{4}_T L^4_{xy}}
& \lesssim &  \| u \|_{L^\infty_T H^{s,0}} +
  \|J_x^{s-} f \|_{L^{4}_T L^1_{xy} }    \label{ImprovedStri1}
\end{eqnarray}
and 
\begin{eqnarray}
\|  J_x^{s-(\frac{2}{3}+)} u \|_{L^{2+}_T L^\infty_{xy}}
& \lesssim &  \| u \|_{L^\infty_T H^{s,0}} +
  \| J_x^s  f \|_{L^{2+}_T L^\frac{4}{3}_{xy} }  \label{ImprovedStri2}
\end{eqnarray}
In particular, any solution  $ u \in L^\infty_T H^{\frac{2}{3}+,0}$  to the KP-I equation belongs to $  L^{2+}_T L^\infty_{xy} $ and satisfies 
\begin{equation}\label{esti1}
\|J_x^{0+}u\|_{L^{2+}_T L^\infty_{xy}} \lesssim  (1+ \|J_x^{\frac{2}{3}+} u \|_{L^\infty_T L^2_{xy}}^2) \|J_x^{\frac{2}{3}+} u \|_{L^\infty_T L^2_{xy}} 
\end{equation}
\end{proposition}

\begin{proof}
Let's first assume \eqref{ImprovedStri1}-\eqref{ImprovedStri2}.  By Lemma \ref{LemProd}  \eqref{prod3} and \eqref{ImprovedStri1} we get that any solution $ u\in L^\infty_T H^{s,0} $ with $s>1/2$  of the KP-I equation satisfies 
$$
\| u \|_{L^{4}_T L^4_{xy}}\lesssim   \| u \|_{L^\infty_T H^{\frac{1}{2}+,0}} +\|u\|_{L^\infty_T L^2_{xy}} \|u\|_{L^\infty_T H^{\frac{1}{2}+,0}}.
$$
In the same way \eqref{ImprovedStri2}  and \eqref{prod3} leads for $ s>2/3 $  to 
$$\| J_x^{s-(\frac{2}{3}+)} u \|_{L^{2+}_T L^\infty_{xy}}
 \lesssim   \| u \|_{L^\infty_T H^{s,0}} +\|u\|_{L^4_T L^4_{xy}} 
  \|u \|_{L^{\infty}_T H^{s,0} }.
$$
Combining this two estimates, we obtain \eqref{esti1}.

It thus remains to prove \eqref{ImprovedStri1}-\eqref{ImprovedStri2}. 
Let $ N \ge 1 $ be a dyadic integer and $ \theta>0 $. We chop the time interval $[0,T]$ into small pieces of length $\sim N^{-\theta} T $, i.e., we define $\{I_{j,N}\}_{j\in J_{N}}$ with $\# J_{N}\sim N^{\theta} $ so that $\bigcup_{j\in J_{N}}I_{j,N}=[0,T]$, $|I_{j,N}|\sim N^{-\theta} T$.
Using the Duhamel formula, \eqref{est1} with $q=4 $ and \eqref{non-hom}  with $(q,r)=(4,4) $ and $(\tilde{q},\tilde{r})=(2+,\infty)$ we get 
\begin{eqnarray}
\| P_N u \|_{L^{4}_T L^4_{xy}}^{4}& \lesssim & \sum_{j\in J_N} \|P_N u \|_{L^\infty_T L^2_{xy}}^{4} +
\sum_{j\in J_N} \Bigl(N^{0+}  \|P_N f_x \|_{L^{2-}_{I_j}  L^1_{xy}} \Bigr)^{4} \nonumber\\
& \lesssim &  \sum_{j\in J_N} \|P_N u \|_{L^\infty_T L^2_{xy}}^{4} +
\sum_{j\in J_N} \Bigl(N^{1+} \, |I_j|^{-1/4+\frac{1}{2}+} \|P_N f \|_{L^{4}_{I_j}  L^1_{xy}} \Bigr)^{4}\nonumber\\
& \lesssim & N^{\theta}  \|P_N u \|_{L^\infty_T L^2_{xy}}^{4} + T^{1+}N^{-\theta+4+} \|P_N f \|_{L^{4}_{T}  L^1_{xy}}^4\; .\label{tyty}
\end{eqnarray}
Proceeding in the same way but without decomposing $ [0,T] $ in small intervals, we  also obtain
$$
\| P_{\le 1} u \|_{L^{4}_T L^4_{xy}} \lesssim   \|P_{\le 1} u \|_{L^\infty_T L^2_{xy}} + T^{\frac{1}{4}+} \|P_{\le 1} f \|_{L^{4}_{T}  L^1_{xy}} \; .
$$
This leads to \eqref{ImprovedStri1} by taking $\theta=2 $ in \eqref{tyty} , applying the obtained estimated to $J^{s-(\frac{1}{2}+)} u $,
 re-summing in $ N\ge 1$ and adding the contribution of $P_{\le 1} u$.
In the same way by applying  this time \eqref{est2} with $ q=2+ $ and \eqref{non-hom} with $(q,r)=(2+,\infty) $ and $(\tilde{q},\tilde{r})=(4,4)$ we get  
\begin{eqnarray}
\| P_N u \|_{L^{2+}_T L^\infty_{xy}}^{2+}& \lesssim & \sum_{j\in J_N} \|J_x^{0+}P_N u \|_{L^\infty_T L^2_{xy}}^{2+} +
\sum_{j\in J_N} \Bigl( N^{0+}  \|P_N f_x \|_{L^{4/3}_{I_j}  L^{4/3}_{xy}} \Bigr)^{2+}\nonumber \\
& \lesssim &  \sum_{j\in J_N}  N^{0+} \|P_N u \|_{L^\infty_T L^2_{xy}}^{2+} +
\sum_{j\in J_N} \Bigl(N^{1+0+} \, |I_j|^{\frac{3}{4}-\frac{1}{2+}} \|P_N f \|_{L^{2+}_{I_j}  L^{4/3}_{xy}} \Bigr)^{2+} \nonumber\\
& \lesssim & N^{\theta+}  \|P_N u \|_{L^\infty_T L^2_{xy}}^{2+} + T^{\frac{1}{2}+}N^{-\frac{\theta}{2}+2+} \|P_N f \|_{L^{2+}_{T}  L^{4/3}_{xy}}^{2+}
\label{toti1}
\end{eqnarray}
that leads to \eqref{ImprovedStri2} by taking $\theta=4/3$, applying the obtained estimated to $J^{s-(\frac{2}{3}+)} u $, re-summing in $ N\ge 1$ and adding the contribution of $ P_{\le 1} u $   given by 
$$
\|  P_{\le 1} u \|_{L^{2+}_T L^\infty_{xy}}
\lesssim    \|P_{\le 1} u \|_{L^\infty_T L^2_{xy}} + T^{\frac{1}{4}+} \|P_{\le 1} f \|_{L^{2+}_{T}  L^{4/3}_{xy}}^{2+}.
$$
\end{proof}

Let $ u $ be a solution to KP and $ N\gg 1$.  We will make a constant use  of the following estimate for $ P_{\lesssim N} u $ on  small time intervals $ I $  of length  $ |I| \lesssim N^{-1} $.

\begin{lemma}\label{lemnew}
Let $ 0<T<2 $ and $ u \in L^\infty_T H^{\frac{3}{4},0}$ be a solution to 
\begin{equation}\label{eqeq}
(u_t + u_{xxx} +f_x)_x \pm  u_{yy} =0 
\end{equation}
then for any $ N\gg 1 $ and any sub-interval of time $ I\subset ]0,1[ $ with $ |I| \lesssim N^{-1} T $ it holds
\begin{eqnarray}
\|  P_{\lesssim N} u \|_{L^{2}_I L^\infty_{xy}}
& \lesssim & T^{\frac{1}{4}-} N^{-\frac{1}{4}+}  \Bigl( \|J_x^{\frac{3}{4}} u \|_{L^\infty_I L^2_{xy}} +
  \|J_x^{\frac{3}{4}}  f \|_{L^{2}_I L^2_{xy} } \Bigr) \label{ImprovedStri5}
\end{eqnarray}
 \end{lemma}
\begin{proof}

Let $t_0 \in I $ with $\|u(t_0)\|_{H^{\frac{3}{4}}} \le \|u\|_{L^\infty_I H^{\frac{3}{4}}} $, we use H\"older inequality in time and  \eqref{est2} with $ q=4- $ for the free evolution and  \eqref{non-hom} with $ (q_1,r_1)=(2+,\infty) $ and $ (q_2,r_2)=(\infty,2) $ for the Duhamel term to get 
\begin{eqnarray}
\|u\|_{L^2_I L^\infty_{xy}} & \lesssim  & |I|^{\frac{1}{4}-} \|U_\pm(\cdot-t_0)u(t_0)\|_{L^{4-}_I L^\infty_{xy}} + \|D_x^{0+}P_{\lesssim N} \partial_x f \|_{L^1_I L^2_{xy}}
\nonumber\\
& \lesssim & |I|^{\frac{1}{4}-}\|u(t_0)\|_{H^{\frac{3}{4},0}} + N^{\frac{1}{4}-} |I|^{1/2} \|P_{\lesssim N}  D_x^{\frac{3}{4}+}f\|_{L^2_I L^2_{xy}}\nonumber \\
&\lesssim & T^{\frac{1}{4}-}N^{-\frac{1}{4}+} \|u(t_0)\|_{H^{\frac{3}{4},0}} + T^\frac{1}{2}N^{-\frac{1}{4}+} \|P_{\lesssim N}  D_x^{\frac{3}{4}+} f\|_{L^2_I L^2_{xy}}\label{estcourt}
\end{eqnarray}
that leads  to \eqref{ImprovedStri5}.
\end{proof}

\section{Trilinear estimates}

In this section, we prove some crucial trilinear estimates which are used in the energy estimates in the consequent sections. 
\begin{definition}
For $u,v\in L^2(\R)$ and $a\in L^\infty(\R^2)$, we set
\begin{equation}\label{def_lambda}
    \F_{x}(\Lambda_a(u,v))(\xi)
    :=\int_{\xi_1+\xi_2=\xi}a(\xi_1,\xi_2)
      \hat{u}(\xi_1)\hat{v}(\xi_2).
\end{equation}
 In the following we will say that $ a $ is acceptable if 
\begin{equation}\label{estcom}
\|\Lambda_a(g,h)\|_{L^2_{x}} \lesssim \|g\|_{L^\infty_x} \|h\|_{L^2_x}\quad  , \forall (g,h)\in ( L^\infty(\R) \cap L^2(\R))^2 
\end{equation}
Note that when $a\equiv1$, we have $\Lambda_a(u,v)=uv$ that obviously satisfies  \eqref{estcom}.
\end{definition}
\begin{definition}
For $u_i\in L^2(\R^3)$ and $a\in L^\infty(\R^2)$, we set
\begin{equation}\label{Pia}
\Pi_a(u_1,u_2,u_3)
    :=\int_{\sum_{i=1}^3 (\tau_i,\xi_i,\eta_i) =(0,0,0)} a(\xi_2,\xi_3)
    \prod_{i=1}^3  \hat{u_i}(\tau_i,\xi_i,\eta_i).
\end{equation}
Note that
\EQ{
\Pi_a(u_1,u_2,u_3)=\iiint u_1(t,x,y) \Lambda_a(u_2(t,\cdot,y),u_3(t,\cdot,y))(x)\, dtdxdy.
}
\end{definition}
It is worth noticing that 
\EQ{
\Pi_a(u_1,u_2,u_3)=\Pi_{\tilde{a}}(u_3,u_1,u_2)=\Pi_{\tilde{\tilde{a}}}(u_2,u_3,u_1)
}
with $  \tilde{a}(z_1,z_2):=a(-z_1-z_2,z_1) $ and $ \tilde{\tilde{a}}(z_1,z_2):=a(z_2,-z_1-z_2) $. Therefore any estimate on $ \Pi_a(u_1,u_2,u_3) $ with a constant that only depends on $ \|a\|_{L^\infty(\R^2)} $ will hold for any permutation of $(u_1,u_2,u_3) $. 

We will apply a trilinear estimate on small  time intervals with length depending on the highest $x$-frequency of the involved functions. The $ X^{0,\frac{1}{2},1}_\pm $ will be only used as intermediate spaces. We need to recall some linear estimates in the context of $ X^{0,\frac{1}{2},1}_\pm $-space on small time interval  (See \cite{IKT})
\begin{lemma}\label{lem1}
Let $ \eta \in C_0^\infty(\R) $ then for any $ u\in X^{0,1/2,1}_\pm  $  and  any $ K>0 $ it holds 
\begin{equation}\label{estXX}
\|\eta(\cdot/K) \, u \|_{X^{0,1/2,1}_\pm} \lesssim \|u \|_{X^{0,1/2,1}_\pm}\quad \text{ and } \quad 
\|\eta(\cdot/K) \, u \|_{L^2_{txy}} \lesssim (1\wedge K^{-1/2}) \|u \|_{X^{0,1/2,1}_\pm}
\end{equation}
\end{lemma}
\begin{lemma}\label{lem2}
Let $ \eta \in C_0^\infty(\R) $. Then for $ \varphi \in L^2_{xy}$ and  any $ K>0 $ it holds 
\begin{equation}\label{linearS}
\| \eta(\cdot/K) \, S(t) \varphi \|_{X^{0,1/2,1}_\pm} \lesssim \|\varphi \|_{L^2_{xy}}
\end{equation}
Moreover, for any $ f\in L^2(\R^3) $ and  any $ K>0 $ it holds 
\begin{equation}\label{linearD1}
\|\eta(\cdot/K) \,\int_0^t S(t-t') f(t') \, dt' \|_{X^{0,1/2,1}_\pm} \lesssim \sum_{L\ge 1} L^{1/2} (L+K)^{-1} \|Q_L f\|_{L^2_{txy}}
\lesssim \|f\|_{X^{0,-1/2,1}_\pm}.
\end{equation}
In particular, for $ K\ge 1$, 
\begin{equation}\label{linearD2}
\| \eta(\cdot/K) \,\int_0^t S(t-t') f(t') \, dt' \|_{X^{0,1/2,1}_\pm} \lesssim K^{-1/2}  \| f\|_{L^2_{txy}}.
\end{equation}
\end{lemma}

We will make a constant use of the following trilinear estimate that follows from Strichartz estimates together with the famous smoothing lemma derived in \cite{IKT} in the case of the KP-I equation and together with the resonance relation in the case of the KP-II equation.
\begin{proposition}\label{Pro1}
Let $ a\in L^\infty(\R^2) $ with $ \|a\|_{L^\infty} \lesssim 1 $.
For any triplet $(N_1,N_2,N_3) \in [1,+\infty[^2\times\R_+^* $ with   $N_1\sim N_2\ge  N_3>0 $ and any $ u_1, \, u_2,\, u_3\in X^{0,1/2,1}_\pm$ where $X^{0,\frac{1}{2},1}_\pm$ is the Bourgain'space defined in \eqref{deftX} associated with the KP-I or KP-II equations,  it holds 
\begin{equation}\label{triestX1}
\Bigl| \Pi_a \Bigl(  P_{N_1} u_1, P_{N_2}  u_2 , P_{N_3}  u_3 \Bigr) \Bigr| 
\lesssim (N_1 N_2 N_3)^{-1/2}  \prod_{i=1}^3 \|P_{N_i} u_i \|_{X^{0,1/2,1}_\pm} \; .
\end{equation}
\end{proposition}
\begin{proof}
We decompose each  function $u_i $, $i=1,2,3 $ as $ u_i =\sum_{L_j\ge 1} Q_j u_i $. Obviously, it suffices to prove that for any triplet $ (L_1,L_2,L_3) \in (2^{\N})^3 $ it holds 
\EQN{
I_{L_1,L_2,L_3}  :=& \Bigl| \Pi_a \Bigl(  Q_{L_1}P_{N_1} u_1, Q_{L_2}P_{N_2}  u_2 , Q_{L_3}P_{N_3}  u_3 \Bigr) \Bigr| \\
\lesssim& (N_1 N_2 N_3)^{-1/2}  \prod_{i=1}^3 \|Q_{L_i}P_{N_i} u_i \|_{X^{0,1/2,1}_\pm} \; .
}
 In the region  $\max(L_1,L_2,L_3) \ge 2^{-8} N_1 N_2 N_3 $ it follows from the  $ L^4 $-Strichartz estimate.  For instance if $L_1\ge  2^{-8}N_1 N_2 N_3 $, it holds
\begin{align*}
I_{L_1,L_2,L_3} 
& \lesssim \int_{\sum_{i=1}^3 (\tau_i,\xi_i,\eta_i) =(0,0,0)}  \prod_{i=1}^3 |\widehat{Q_{L_i} P_{N_i} u_i}|\\
& \le \|Q_{L_1} P_{N_1} u_1\|_{L^2_{txy}} \| {\mathcal F}_{txy}^{-1}(|\widehat{Q_{L_2} P_{N_2} u_2}|)\|_{L^{4}_{txy}} 
 \|{\mathcal F}_{txy}^{-1}(|\widehat{Q_{L_3}  P_{N_3} u_3}|)\|_{L^{4}_{txy}} 
  \\
  & \lesssim (N_1 N_2 N_3)^{-1/2}    \|Q_{L_1} P_{N_1} u_1\|_{X^{0,1/2,1}_\pm}
  \| Q_{L_2} P_{N_2} u_2\|_{X^{0,1/2,1}_\pm} \| Q_{L_3} P_{N_2} u_2\|_{X^{0,1/2,1}_\pm}
 \end{align*}
 and the cases $ L_2$ or $ L_3 \ge  2^{-8} N_1 N_2 N_3 $ follows the same lines. 
 
The estimate $\max(L_1,L_2,L_3) < 2^{-8} N_1 N_2 N_3 $ cannot hold for the KP-II equation thanks to the well-known  resonance relation
$$
\sigma_+(\tau_1,\xi_1,\mu_1)+\sigma_+(\tau-\tau_1,\xi-\xi_1,\mu-\mu_1)-\sigma_+(\tau,\xi,\mu)= \xi \xi_1 (\xi-\xi_1) \Bigl( 3+ \frac{(\mu \xi_1 -\mu_1 \xi)^2}{( \xi \xi_1 (\xi-\xi_1))^2}\Bigr). 
$$
Finally, for the KP-I equation,  in the region $\max(L_1,L_2,L_3) < 2^{-8} N_1 N_2 N_3 $, the result follows from the the famous smoothing estimate of Ionescu-Kenig -Tataru (\cite{IKT}, Lemma 5.2 page 283).
\end{proof}

\begin{proposition}[Refined Estimate I]\label{prop_tri}
Let $ \epsilon\in \{-1,1\}$, $0<T<2$, and  $N_1\sim N_2\gtrsim  N_3>0$ with $ N_1\gg 1$.
Let  $f_1,f_2, f_3\in L^2(]0,T[; L^2(\R^2))$ and let $a\in L^\infty(\R^2)$ with  $\|a\|_{L^\infty}\lesssim 1$ that is acceptable in the sense of \eqref{estcom}.
Finally let   $u_1,u_2, u_3\in C([0,T];L^2(\R^2))$ satisfying
\begin{equation}\label{equj}
\partial_x (\partial_t u_j + \partial_x^3 u_j +\partial_x f_j) +\epsilon \partial_{yy}  u_j=0
\end{equation}
on $]0,T[\times \R^2$ for $j=1,2,3$. Then for any decomposition of $ P_{N_3} f_3$ as 
$P_{N_3} f_3=P_{N_3} f_{3,1} + P_{N_3}  f_{3,2}$ with $ f_{3,1}\in  L^2(]0,T[; L^2(\R^2))$ and $ f_{3,2}\in  L^\infty(]0,T[; L^{1+}(\R^2))$ it holds 
\EQ{
&\Bigl| \int_0^T  \int_{\R^2} P_{N_1} u_1(\tau)  \Lambda_a(P_{N_2} u_2(\tau) , P_{N_3} u_3(\tau)) \, d\tau \Bigr| \\
\lesssim& \; \Bigl(     N_3^{-1/2} \|P_{N_3} u_3 \|_{L^\infty_T L^2_{xy}}+\Bigl(\frac{N_3}{N_1}\Bigr)^{\frac{1}{2}} \cro{N_3}^{0+}\sup_{I\subset [0,T] \atop |I|\lesssim N_1^{-1}T} \Bigl[ \|P_{N_3} f_{3,1} \|_{L^2_I L^2_{xy}} +  \|P_{N_3} f_{3,2} \|_{L^{\infty-}_I L^{1+}_{xy}}\Bigr]\Bigr)\\
& \times\prod_{i=1}^2 \Bigl( T^{-1/2} \| P_{N_i} u_i\|_{L_T^2 L^2_{xy}}+  \Bigl(\frac{N_3}{N_1}\Bigr)^{\frac{1}{2}}
 \| P_{N_i} u_i\|_{L_T^\infty  L^2_{xy}}+T^\frac{1}{2}\|P_{N_i} f_i \|_{L_T^2 L^2_{xy}}\Bigr)\; .\label{tritri}
}
\end{proposition}

\begin{proof}
Let $ \eta\in C^\infty_0(\R) $ with values in $ \R_+ $ with $\supp \eta \in [-3/4,3/4] $, $\eta\equiv 1$ on $[-1/4,1/4] $ and $\eta+\eta(\cdot-1)\equiv 1 $ on $ [1/4,3/4] $. Then by construction 
$$
\sum_{k\in\Z} \eta(\cdot-k) \equiv 1 \quad \text{ on } \R \; .
$$
For $ K\ge 4 $ we set $ \eta_{K} = \eta(\cdot/K) $,  $ \eta_{K,j}:=  \eta_{K}(\cdot-K j) $ and $\tilde{\eta}_{K,j}:=  \eta_{K/4}(\cdot-Kj) $ so that 
 $\eta_{K,j} \tilde{\eta}_{K,j} =\eta_{K,j} $.
For $T>0 $ and $ N_1\ge 1 $ a dyadic integer we notice that 
\begin{equation}\label{sum}
\sum_{k=0}^{N_1}  \eta_{\frac{N_1}{T},k}\equiv 1 \quad \text{ and }  \sum_{k=0}^{N_1}  \tilde{\eta}_{\frac{N_1}{T},k}\le 6 \quad \text{ on } [0,T]  
\end{equation}
with 
$\supp \eta_{\frac{N_1}{T},k}\subset [0,T] $ for any $ 1\le k\le N_1-1 $. More precisely it holds 
\begin{equation}\label{OT}
1\! \!1_{[0,T]} = 1\! \!1_{[0,T]}\Bigl(\eta_{\frac{N_1}{T},0}+ \eta_{\frac{N_1}{T},N_1}\Bigr)+ \sum_{k=1}^{N_1-1} \eta_{\frac{N_1}{T},k}\; .
\end{equation}

We defined the subintervals $I_j=I_{j,N_1}$ , $1\le j\le N_1-1 $, of $ [0,T] $ by 
$$
 I_j:=[TN_1^{-1}(j-1/2), T N_1^{-1}(j+1/2)]  \quad \text{ for } 1\le j\le N_1-1\; .
$$
By construction we have $\displaystyle\cup_{j=1}^{N_1} I_j \subset [0,T] $.
For $1\le j \le N_1-1$ and $i\in \{1,2\} $, we choose $c_{i,j,N_1}\in I_J $  at which $\|P_{N_i}u_i(t)\|_{L_x^2}^2$ attains its minimum on $I_j $. In particular, for any $1\le j \le N_1-1$ and any $ i\in \{1,2\} $ it holds 
\begin{equation}\label{defcij}
   \| P_{N_i} u_i(c_{i,j})\|_{L^2_{xy}}^2\le \frac{1}{|I_j|}  \| P_{N_i} u_i\|_{L^2_{I_j}L^2_{xy}}^2
   \lesssim \frac{N_1}{T} \| P_{N_i} u_i\|_{L^2_{I_j}L^2_{xy}}^2
  \end{equation}
  To simplify the notations, we write $c_{i,j}=c_{i,j,N_1}$ and $ \eta_k =\eta_{\frac{N_1}{T},k}$. According to \eqref{sum},  it holds 
  \begin{align}
  \int_0^T & P_{N_1} u_1(\tau)  \Lambda_a(P_{N_2} u_2(\tau) , P_{N_3} u_3(\tau)) \, d\tau \nonumber \\
  &= \Pi_a(1\! \!1_{[0,T]} \eta_0 P_{N_1} u_1, P_{N_2} u_2, P_{N_3} u_3)+  \Pi_a(1\! \!1_{[0,T]} \eta_{N_1}  P_{N_1} u_1,  P_{N_2} u_2, P_{N_3} u_3)\nonumber \\
 & + \sum_{j=1}^{N_1-1}  \Pi_a( \eta_{j} P_{N_1} u_1,  P_{N_2} u_2, P_{N_3} u_3)\nonumber \\
  &=\Pi_a(1\! \!1_{[0,T]} \eta_0 P_{N_1} u_1, \tilde{\eta}_0  P_{N_2} u_2,  \tilde{\eta}_0 P_{N_3} u_3) +  \Pi_a(1\! \!1_{[0,T]} \eta_{N_1}  P_{N_1} u_1,  \tilde{\eta}_{N_1}  P_{N_2} u_2,  \tilde{\eta}_{N_1} P_{N_3} u_3)\nonumber \\
  & + \sum_{j=1}^{N_1-1}  \Pi_a( \eta_{j} P_{N_1} u_1, \tilde{\eta}_j  P_{N_2} u_2,  \tilde{\eta}_j P_{N_3} u_3)\nonumber \\
  &:= A+B+C \; .
  \end{align}
   We start by estimating the contribution of $ C $. Our main tools will be Proposition \ref{Pro1} together with Lemmas \ref{lem1}-\ref{lem2}.
   
 {\it Contribution of $ C $.}
 Since $u_i$ satisfies \eqref{equj}, for any $ t\in [0,T] $ we have 
  \begin{equation}\label{interval}
   u_i(t)=    U_{\pm}(t-c_{i,j})P_{N_i} u_i(c_{i,j})
     +  P_{N_i} F_{i,j} (t),\quad
   \end{equation}
  where
    \begin{equation}
    F_{i,j}(t)
    :=\int_{c_{i,j}}^t U_\pm (t-\tau) \partial_x f_{i}(\tau) d\tau   \;.
   \end{equation}
   In a first step we  only substitute $ u_1 $ and $ u_2 $ by its integral form \eqref{interval}.
   Since  $ N_1\sim N_2 $,  by symmetry it suffices to estimate the 3 following contributions : 
   \begin{eqnarray*}
  C_{1} &=&  \displaystyle \sum_{j=1}^{N_1-1}\Pi_a\Bigl(\eta_j U_\pm(t-c_{1,j})P_{N_1} u_1(c_{1,j}),
    \tilde{\eta}_j U_\pm (t-c_{2,j})P_{N_2} u_2(c_{2,j}), \tilde{\eta}_j P_{N_3} u_3\Bigr)\\
      C_{2} &=&  \displaystyle \sum_{j=1}^{N_1-1}\Pi_a\Bigl(\eta_j P_{N_1} F_{1,j},
    \tilde{\eta}_j P_{N_2}F_{2,j},  \tilde{\eta}_j  P_{N_3} u_3\Bigr) \\
  C_{3} &=&  \displaystyle \sum_{j=1}^{N_1-1}\Pi_a\Bigl(\eta_j U_\pm (t-c_{1,j})P_{N_1} u_1(c_{1,j}),
     \tilde{\eta}_j P_{N_2} F_{2,j}, \tilde{\eta}_j P_{N_3} u_3\Bigr)\\
        \; .
    \end{eqnarray*}
    {\bf Contribution of $ C_1$ :} In  view of \eqref{triestX1}, \eqref{estXX}, \eqref{linearS} and \eqref{defcij}, performing a Cauchy-Schwarz in $ j $
          it holds 
    \begin{eqnarray*}
   | C_1| & \lesssim &(N_1 N_2 N_3)^{-1/2}\sum_{j=1}^{N_1-1} \|  \tilde{\eta}_j P_{N_3} u_3\|_{X^{0,\frac{1}{2},1}_\pm}\prod_{i=1}^2 \| P_{N_i} u_i(c_{i,j})\|_{L^2_{xy}} \\
    & \lesssim & T^{-1} N_3^{-1/2} ( \sup_{j\in \{1,..,N_1-1\}}\|  \tilde{\eta}_j P_{N_3} u_3\|_{X^{0,\frac{1}{2},1}_\pm} )
    \prod_{i=1}^2 \Bigl( \frac{T}{N_1} \sum_{j=1}^{N_1-1} \| P_{N_i} u_i(c_{i,j})\|_{L^2_{xy}}^2 \Bigr)^{1/2}\\
   & \lesssim &   T^{-1}  N_3^{-1/2} ( \sup_{j\in \{1,..,N_1-1\}}\|  \tilde{\eta}_j P_{N_3} u_3\|_{X^{0,\frac{1}{2},1}_\pm} )
   \prod_{i=1}^2  \Bigl(  \sum_{j=1}^{N_1-1} \| P_{N_i} u_i\|_{L^2_{I_j}L^2_{xy}}^2 \Bigr)^{1/2} \\
    & \lesssim &    N_3^{-1/2}  ( \sup_{j\in \{1,..,N_1-1\}}\|  \tilde{\eta}_j P_{N_3}    u_3\|_{X^{0,\frac{1}{2},1}_\pm} )  \prod_{i=1}^2  (T^{-\frac{1}{2}}\|P_{N_i} u_i \|_{L^2_{T} L^2_{xy}}) \; .
    \end{eqnarray*} 
       {\bf Contribution of $ C_2$ :} We  use this time \eqref{linearD2}  and  \eqref{triestX1}  and perform the same Cauchy-Schwarz inequality in $ j $  
to get  
    \begin{eqnarray*}
    |C_2| & \lesssim &  (N_1 N_2 N_3)^{-1/2} \sum_{j=1}^{N_1-1} 
    \|  \tilde{\eta}_j P_{N_3} u_3\|_{X^{0,\frac{1}{2},1}_\pm} \prod_{i=1}^2  \Bigl( T^{1/2} N_1^{-1/2} \| \tilde{\eta}_j  P_{N_i} \partial_x f_{i,j} \|_{L^2_{txy}}\Bigr)\\
    & \lesssim & T  N_3^{-1/2} ( \sup_{j\in \{1,..,N_1-1\}}\|  \tilde{\eta}_j P_{N_3} u_3\|_{X^{0,\frac{1}{2},1}_\pm} )
    \prod_{i=1}^2 \Bigl(\sum_{j=1}^{N_1-1}  \|\tilde{\eta}_j  P_{N_i}   f_i \|_{L^2_t L^2_{xy}}^2 \Bigr)^{1/2}\\
   & \lesssim &   N_3^{-1/2}  ( \sup_{j\in \{1,..,N_1-1\}}\|  \tilde{\eta}_j P_{N_3} u_3\|_{X^{0,\frac{1}{2},1}_\pm} )
    \prod_{i=1}^2  (T^{\frac{1}{2}}\| P_{N_i}  f_i \|_{L^2_T L^2_{xy}}) \quad . 
    \end{eqnarray*} 
    where we made use of the second part of \eqref{sum}. \\
 {\bf Contribution of $ C_3$ :} We combine the arguments used to control $ C_1$  and $C_2$   
 to get  
\begin{eqnarray*}
|C_3| & \lesssim & (N_1 N_2 N_3)^{-1/2}  ( \sup_{j\in \{1,..,N_1-1\}}\|  \tilde{\eta}_j P_{N_3} u_3\|_{X^{0,\frac{1}{2},1}_\pm} )\\&&\sum_{j\in J_{N_1}}
\| P_{N_1} u_1(c_{1,j})\|_{L^2_{xy}}  ( T^{1/2} N_1^{-1/2} \| \tilde{\eta}_j  P_{N_i} \partial_x f_{2,j} \|_{L^2_{txy}}) \\
 & \lesssim &   N_3^{-1/2}  ( \sup_{j\in \{1,..,N_1-1\}}\|  \tilde{\eta}_j P_{N_3} u_3\|_{X^{0,\frac{1}{2},1}_\pm} )
\Bigl( \sum_{j=1}^{N_1-1} T N_1^{-1} \| P_{N_1} u_1(c_{1,j})\|_{L^2_{xy}}^2 \Bigr)^{1/2}\\&&
\Bigl(\sum_{j=1}^{N_1-1} \| \tilde{\eta}_j   P_{N_2} f_2 \|_{L^2_t L^2_{xy}}^2 \Bigr)^{1/2}\\
& \lesssim &N_3^{-1/2}  ( \sup_{j\in \{1,..,N_1-1\}}\|  \tilde{\eta}_j P_{N_3} u_3\|_{X^{0,\frac{1}{2},1}_\pm} ) (T^{-\frac{1}{2}} \|P_{N_1} u_1 \|_{L^2_{T} L^2_{xy}})
(T^{\frac{1}{2}}\|P_{N_2} f_2 \|_{L^2_{T} L^2_{xy}})
 \; .
\end{eqnarray*} 
Gathering the contributions of $C_1, \, C_2 $ and $ C_3$ we obtain
\begin{equation}\label{nana}
    |C|\lesssim N_3^{-1/2}( \sup_{j\in \{1,..,N_1-1\}}\!\!\!\!\|  \tilde{\eta}_j P_{N_3} u_3\|_{X^{0,\frac{1}{2},1}_\pm} ) \prod_{i=1}^2  (T^{-\frac{1}{2}}\|P_{N_i} u_i\|_{ L^2_T L^2_{xy}} +T^{\frac{1}{2}}\| P_{N_i}  f_i \|_{L^2_T L^2_{xy}})\, .
\end{equation}
It remains to estimate $ N_3^{-1/2} \sup_{j\in \{1,..,N_1-1\}}\|  \tilde{\eta}_j P_{N_3} u_3\|_{X^{0,\frac{1}{2},1}_\pm}$.   According to \eqref{hom} with $ \varepsilon=0+$, $q=2$ and $ r=+\infty $ and classical injection of the Strichartz estimates in Bourgain'spaces (see for instance \cite{ginibre}) it holds 
    $$
    \|w\|_{L^{2+}_t L^\infty_{xy}} \lesssim \|D_x^{0+} w\|_{X^{0,1/2,1}_\pm} 
    $$
    and by duality and interpolation we infer that 
\begin{equation}
     \| w\|_{X^{0,-1/2,1}_\pm} \lesssim  \|D_x^{0+} w\|_{L^{2-}_t L^{1+}_{xy}}
\end{equation}
    that leads to 
\EQN{
\| \tilde{\eta}_j  P_{N_3} \partial_x f_{3,2} \|_{X^{0,-1/2,1}_\pm} \lesssim &\|\tilde{\eta}_j P_{N_3}D_x^{0+} \partial_x f_{3,2} \|_{L^{2-}_t  L^{1+}_{xy}}\\
\lesssim&  T^{\frac{1}{2 }} N_1^{-\frac{1}{2}} N_3  \|\tilde{\eta}_j P_{N_3} D_x^{0+}f_{3,2} \|_{L^{\infty-}_{t} L^{1+}_{xy}}. 
}
Therefore substituting $ u_3 $ by its integral form  \eqref{interval} and making use of \eqref{linearS} and \eqref{linearD2}  we get for any 
$  j\in \{1,..,N_1\} $,
\EQN{
N_3^{-1/2}\|  \tilde{\eta}_j P_{N_3} u_3\|_{X^{0,\frac{1}{2},1}_\pm} \lesssim& N_3^{-1/2} \|P_{N_3} u_3 \|_{L^\infty_T L^2_{xy}}\\
&+T^\frac{1}{2}\Bigl(\frac{N_3}{N_1}\Bigr)^{\frac{1}{2}} \cro{N_3}^{0+}\sup_{I\subset [0,T] \atop |I|\lesssim N_1^{-1}T} \Bigl[ \|P_{N_3} f_{3,1} \|_{L^2_I L^2_{xy}} +  \|P_{N_3} f_{3,2} \|_{L^{\infty-}_I L^{1+}_{xy}}\Bigr].
}
This last estimate together  with  \eqref{nana} leads to an admissible estimate on $ C$.\vspace*{2mm}

 {\it Contribution of $ A $.} We notice that we can not use Proposition \ref{pro1} to estimate the contribution of $ A$  since $1_{[0,T]}\not\in B^{1/2,1}_2(\R) $.
  We will instead make use of the Strichartz estimates stated in Lemma \ref{Strichartz}  together with \eqref{estcom}. We apply  \eqref{estcom} and Holder inequality in $ y $ and  then $t$ to get 
    \begin{align*} 
|A|& = \Bigl|\int_0^T \int_{\R^2}\eta_0 P_{N_1} u_1(t) \Lambda_a( \tilde{\eta}_0 P_{N_2} u_2 (t), \tilde{\eta}_0 P_{N_3} u_3(t))\; dt dx dy  \Bigr|\\
&\lesssim  \|\eta_0  P_{N_1} u_1\|_{L^2_{Txy}} \|\tilde{\eta}_0 P_{N_2} u_2\|_{L^\infty_{T} L^2_{xy}}
 \|\tilde{\eta}_0 P_{N_3} u_3\|_{L^2_{T} L^{\infty}_{xy}} \\
 & \lesssim  T^{1/2}N_1^{-1/2}  \| P_{N_1} u_1\|_{L^\infty_{T} L^2_{xy}} \| P_{N_2} u_2\|_{L^\infty_{T} L^2_{xy}}
 \|\tilde{\eta}_0 P_{N_3} u_3\|_{L^2_{T} L^{\infty}_{xy}}.
\end{align*}
Then we rewrite $u_3$,  as 
$
u_3(t)= S(t) u_3(0) + \int_0^t S(t-\tau) u_3(\tau) \,d\tau 
$
and use Lemma \ref{Strichartz} to get 
\begin{align*}
 \|\tilde{\eta}_0 P_{N_3} u_3\|_{L^2_{T} L^{\infty}_{xy}}& \lesssim \|P_{N_3} u_3(0)\|_{L^2_x} + \|\tilde{\eta}_0 D_x^{0+} \partial_x P_{N_3} f_{3,1}\|_{L^1_T L^2_{xy}}
 +\|\tilde{\eta}_0  D_x^{0+} \partial_x P_{N_3} f_{3,1}\|_{L^{2-}_T L^{1+}_{xy}}\nonumber \\
  &  \lesssim \|P_{N_3} u_3(0)\|_{L^2_x} \\
  &\quad+N_3^{1+} N_1^{-1/2} T^{1/2}\sup_{I\subset [0,T] \atop |I|\lesssim N_1^{-1}T} 
  \Bigl[ \|P_{N_3} f_{3,1} \|_{L^2_I L^2_{xy}}
 +   \|P_{N_3} f_{3,2} \|_{L^{\infty-}_I L^{1+}_{xy}} \Bigr]\; .
  \end{align*}
Since $t\mapsto P_{N_3} u(t) $ is continuous with values in $ L^2(\R^2) $ we eventually obtain 
\EQ{
|A|  &\lesssim T^{\frac{1}{2}} \Bigl( \frac{N_3}{N_1}\Bigr)^{\frac{1}{2}} \prod_{i=1}^2 \|P_{N_i} u_i \|_{L^\infty_T L^2_{xy}} \Bigl( N_3^{-\frac{1}{2}} 
\| P_{N_3} u_3 \|_{L^\infty_T L^2_{xy}} \\
& \hspace*{10mm} +
  T^{\frac{1}{2}} \Bigl( \frac{N_3}{N_1}\Bigr)^{\frac{1}{2}} \cro{N_3}^{0+}\sup_{I\subset [0,T] \atop |I|\lesssim N_1^{-1}T}  \Bigl[ \|P_{N_3} f_{3,1} \|_{L^2_I L^2_{xy}} +
    \|P_{N_3} f_{3,2} \|_{L^{\infty-}_I L^{1+}_{xy}}\Bigr]\Bigr)
 }
which is admissible. Finally, $B$ can be treated exactly in the same way by rewriting $ u_3 $ as $
u_3(t)= U_\pm (t) u_3(T) + \int_T^t U_\pm(t-\tau) u_3(\tau) \,d\tau 
$.
\end{proof}

\section{A priori estimates on the solutions}

In this section, we gather all the estimates of the previous sections to  derive a-priori estimates in $L^\infty_T H^{s,0}$. First we prove

\begin{proposition} Let  $0<T<2$, $s>3/4$ and  $u\in L^\infty_T H^{s,0}$ be a solution of the KP-I equation then it holds 
\EQ{\label{estim}
\|u\|_{L^\infty_T H^{s,0}_\om}^2& \lesssim  \|u_0\|_{H^{s,0}_\om}^2+ T^{1/2} \|u \|_{L^2_T L^\infty_{xy}} \| u\|_{L^\infty_T L^2_{xy}}^2 \\
&\quad + \| u\|_{L_T^\infty H^{s,0}_\om}^2 \|u\|_{L^\infty_T H^{\frac{3}{4}+,0}}(1+ \|u\|_{L^2_T L^\infty_{xy}}^2)(1+ \|u\|_{L^2_T L^\infty_{xy}}^2+\|u\|_{L^\infty_T H^{\frac{3}{4}+,0}}^2).
}
\end{proposition}
\begin{proof}
First we notice that, according to \eqref{esti1}, $ u\in L^2_T L^\infty_{xy} $. 
By the energy estimates, we have
\EQ{
\|u(t)\|_{H^{s,0}_\om}^2\les& \sum_{N} \om_N^2 \jb{N}^{2s} \|P_Nu(t)\|_{L^2}^2\\
\les& \sum_{N}\om_N^2 \jb{N}^{2s}\|P_Nu(0)\|_{L^2}^2+\sum_{N}\om_N^2 \jb{N}^{2s} \int_0^t  \int_{\R^2} \partial_x P_N(u^2) P_N u dsdxdy.
}

For $ N\lesssim 1 $, we have
\EQ{\label{low}
&\Bigl|\sum_{0<N\lesssim 1}\om_N^2 \langle  N\rangle^{2s}\int_0^t  \int_{\R^2} \partial_x P_N(u^2) P_N u dsdxdy\Bigr| \\
\lesssim & T^{1/2} \|P_{\lesssim 1} u\|_{L^\infty_T L^2_{xy}}  \|P_{\lesssim 1} (u^2)\|_{L^2_T L^2_{xy}} 
\sum_{0<N\lesssim 1} N \\
\lesssim& T^{1/2} \|u \|_{L^2_T L^\infty_{xy}} \| u\|_{L^\infty_T L^2_{xy}}^2
}
which is acceptable for $ s>2/3$ thanks to \eqref{esti1}.

For $ N\gg 1 $ we have the two following contributions denoted by $A$ and $B$, where
\EQ{
A=& \sum_{N \gg  1}\om_N^2 \langle  N\rangle^{2s}\int_0^t  \int_{\R^2} \partial_x P_N(u P_{\ll N} u ) P_N u \, dsdxdy\\
 =&  \sum_{N \gg  1}\sum_{0<N_3\ll N}\om_N^2  \langle  N\rangle^{2s}\Bigl( -\int_0^t \int_{\R^2} \partial_x P_{N_3}u  (P_N u)^2\, dsdxdy\\
 &+\int_0^t \int_{\R^2} [\partial_x P_N, P_{N_3}u] \tilde{P}_N u \cdot P_N u  dsdxdy\Bigr) \\
 = &  \sum_{N \gg  1}\sum_{0<N_3\ll N} \om_N^2  \langle  N\rangle^{2s} N_3 \int_0^t \int_{\R^2} \Lambda_{a_1}(P_{N_3}u, \tilde{P}_N u) \cdot P_N u  \, dsdxdy \label{intep}
}
with 
\begin{equation}\label{a1}
a_1(\xi_1,\xi_2)=N_3^{-1}\phi_{N_3}(\xi_1)\tilde \phi_N(\xi_2)[\phi_N(\xi_1+\xi_2)(\xi_1+\xi_2)-\phi_N(\xi_2)\xi_2] \; ,
\end{equation}
 and
\EQ{
B= &  \sum_{N \gg 1}\sum_{N_1\gtrsim N}\om_N^2  \langle  N\rangle^{2s}\int_0^t  \int_{\R^2} \partial_x P_N( \tilde{P}_{N_1}u P_{N_1} u) P_N u dsdxdy\\
= &  \sum_{N_1 \gg 1}\sum_{1\ll N \lesssim N_1} \om_N^2 \langle  N\rangle^{2s} N \int_0^t  \int_{\R^2}\Lambda_{a_2}(P_N u,\tilde{P}_{N_1}  u) P_{N_1} u  dsdxdy\\
\lesssim & \sum_{N_1 \gg 1}\sum_{1\ll N \lesssim N_1}\om_{N_1}^2  \langle  N_1\rangle^{2s} N \int_0^t  \int_{\R^2}    \Lambda_{a_2}(P_N u, \tilde{P}_{N_1}u) P_{N_1} u \, dsdxdy
}
with $a_2(\xi,\xi_1)=\frac{\xi}{N}\phi_{N}(\xi)$.  It is direct to check that $\|a_2\|_{L^\infty(\R^2)} \lesssim 1 $ and that $ a_2 $ satisfies \eqref{estcom}. Moreover, it is well-known (see for instance \cite{KoTz}) that, for $ N_3\ll N$, the following commutator estimate holds
$$
\|\Lambda_{a_1}(g,h)\|_{L^2_x}=N_3^{-1}  \|[\partial_x P_N, P_{N_3}g] \tilde{P}_N h\|_{L^2_x}\lesssim \|g\|_{L^\infty_x} \|h\|_{L^2_x} \;, \quad \forall (g,h)\in L^\infty(\R) \times L^2(\R) ,
$$
and by the mean value theorem, 
$$
\|a_1(\xi_1,\xi_2)\|_{L^\infty(\R^2}\lesssim N_3^{-1} \phi_{N_3}(\xi_1) |\xi_1| \sup_{\R} \Bigl(\phi_N+|\phi'(\frac{\cdot}{N})(\frac{\cdot}{N})| \Bigr)\lesssim 1 \; .
$$
Therefore  it  suffices to bound
\begin{equation}\label{DefC}
C=\sum_{N \gg  1}\sum_{0<N_3\lesssim N}\om_N^2 \langle  N\rangle^{2s} N_3 \int_0^t \int_{\R^2} \Lambda_a( P_{N_3}u,  \tilde{P}_N u) P_N u dsdxdy
\end{equation}
for $ a\in L^\infty(\R^2) $ with $ \|a\|_{L^\infty(\R^2)}\lesssim 1$ satisfying \eqref{estcom}. 
To bound  $C $ we  notice that 
$$ P_{N_3} (u^2) = P_{N_3} \Bigl(2  P_{\ll N_3}u P_{\sim N_3} u + \sum_{N_4\gtrsim N_3} P_{N_4} u P_{\sim N_4} u \Bigr) $$  and apply   \eqref{tritri} with $ f_{3,1}= 2 P_{\ll N_3}u P_{\sim N_3} u $ and $ f_{3,2}= \sum_{N_4\gtrsim N_3} P_{N_4} u P_{\sim N_4}u $  to get 
\EQN{
|C|\lesssim &  \sum_{N \gg  1}\sum_{0<N_3\ll N} 
\Bigl( T^{-1}\|\tilde{P}_{N} u\|_{L_T^2 H^{s,0}_\om}^2+  \Bigl(\frac{N_3}{N}\Bigr)
 \| \tilde{P}_{N} u\|_{L_T^\infty  H^{s,0}_\om}^2+\|\tilde{P}_{N} (u^2)\|_{L_T^2 H^{s,0}_\om}^2\Bigr)N_3^{1\over 2} \\
 &\times  \Bigl[ \|P_{N_3} u \|_{L^\infty_T H^s_{xy}}  \\
 &+    ( \frac{N_3}{N})^{\frac{1}{2}}N_3^{\frac{1}{2}}\cro{N_3}^{0+}
\sup_{I\subset ]0,T[ \atop |I|\lesssim N^{-1}}  \Bigl( \|P_{\ll N_3} u P_{\sim N_3} u  \|_{L^{2}_I L^2_{xy}} 
+ \sum_{N_4\gtrsim N_3} \| P_{N_4} u P_{\sim N_4}u\|_{L^{\infty-}_I L^{1+}_{xy}} \Bigr)  \Bigr].
}
We first notice that for $ N_3\ll N $ it holds  $(N_3/N)\le N^{-(0+)} N_3^{0+}$ 
 and that  H\"older  inequalities lead to 
\EQ{
N_4^{\frac{1}{2}} \|P_{N_4} u \|_{L^{\infty-}_I L^{2+}_{xy}} & \lesssim  N_4^{\frac{1}{2}}\|P_{N_4} u \|_{L^{\infty}_I L^2_{xy}}^{1-} \|P_{N_4} u \|_{L^{2}_I L^\infty_{xy}}^{0+} \\
& \lesssim 1+\cro{N_4}^{\frac{1}{2}+}\|P_{N_4} u \|_{L^{\infty}_I L^2_{xy}}(1+  \|P_{N_4} u \|_{L^{2}_I L^\infty_{xy}}) \label{311}  \; .
}
Hence, 
\EQ{
N_3  \sum_{N_4\gtrsim N_3} \cro{N_3}^{0+} \| P_{N_4} u P_{\sim N_4}u\|_{L^{\infty-}_I L^{1+}_{xy}} \lesssim & 
   \sum_{N_4\gtrsim N_3} \cro{N_4}^{\frac{1}{2}+}  \|  P_{N_4}u\|_{L^\infty_I L^2_{xy}} \cro{N_4}^{\frac{1}{2}}\|P_{N_4} u\|_{L^{\infty-}_I L^{2+}_{xy}}  \\
   \lesssim & \hspace*{2mm} \| u\|_{L^\infty_I H^{\frac{1}{2}+}_{xy}}\Bigl( 1+\| u\|_{L^\infty_I H^{\frac{1}{2}+}_{xy}}(1+\|u\|_{L^2_I L^\infty_{xy}})\Bigr) \label{31}\; .
}
On the other hand,   Lemma \ref{lemnew} leads to 
\EQ{\label{32}
&\cro{N_3}^{0+}N_3  \sup_{I\subset ]0,T[ \atop |I|\lesssim N^{-1}}  \|P_{\ll N_3} u P_{\sim N_3} u  \|_{L^{2}_I L^2_{xy}} \\
\lesssim& \cro{N_3}^{0+}  N_3 \sup_{I\subset ]0,T[ \atop |I|\lesssim N^{-1}}  \|P_{\ll N_3} u \|_{L^2_I L^\infty_{xy}} \|P_{\sim N_3} u  \|_{L^{\infty}_I L^2_{xy}}\\
\lesssim & \Bigl(\frac{N_3}{N}\Bigr)^{\frac{1}{4}-}   \Bigl( \|J_x^{\frac{3}{4}} u \|_{L^\infty_I L^2_{xy}} +\|J_x^{\frac{3}{4}} ( u^2)\|_{L^{2}_I L^2_{xy} }\Bigr) \|u  \|_{L^{\infty}_I H^{\frac{3}{4}+,0}}.
}
Combining the four last estimates we eventually obtain
\EQ{\label{desdes}
|C| \lesssim &  
\Bigl( T^{-1}\|u\|_{L_T^2 H^{s,0}_\om}^2+\|u\|_{L^\infty_T H^{s,0}_\om}^2(1+\|u\|_{L^2_T L^\infty_{xy}}^2)\Bigr)\\
&\times\Bigl(1+  \|u\|_{L^\infty_T H^{\frac{3}{4}+,0}_{xy}}(1+\|u\|_{L^2_T L^\infty_{xy}}) \Bigr)\|u\|_{L^\infty_T H^{\frac{3}{4}+,0}_{xy}}\\
\lesssim &   \| u\|_{L_T^\infty H^{s,0}_\om}^2(1+ \|u\|_{L^2_T L^\infty_{xy}}^2)
\Bigl(1+  \|u\|_{L^\infty_T H^{\frac{3}{4}+,0}_{xy}}(1+\|u\|_{L^2_T L^\infty_{xy}}) \Bigr) \|u\|_{L^\infty_T H^{\frac{3}{4}+,0}_{xy}} 
}
where we made use of \eqref{prod33} to get
\begin{equation}
\|u^2\|_{L_T^2 H^{s,0}_\om}\lesssim \|u\|_{L_T^\infty H^{s,0}_\om} \|u\|_{L^2_T L^\infty_{xy}} \;.\label{der}
\end{equation}
This completes the proof of the proposition.
\end{proof}

\section{Estimates on the difference}
We estimate the difference $w=u_1-u_2$ of two solutions in $ \overline{H}^{s-1,0} $ where 
$$
\|w\|_{\overline{H}^{s-1,0}}\sim \sum_{N>0} \langle N\rangle^{s-1}  \langle N^{-1}\rangle^{3/4}\|P_N w\|_{L^2_{xy}} \; .
$$
Note that $w$ solves the following equation
\begin{eqnarray}\label{eq:kpI-diff}
\begin{cases}
\Bigl( w_t+w_{xxx}+((u_1+u_2)w/2)_x\Bigr)_x +\epsilon w_{yy}=0;\\
w(x,y,0)=w(0).
\end{cases}
\end{eqnarray}
We will  need a slight refinement of  Propositon \ref{prop_tri}. 
\begin{proposition}[Refined Estimate II]\label{propro2}
Under the same hypotheses as in Proposition \ref{prop_tri}  for any decomposition of $  f_i$, $i=1,2$, as 
$f_i= f_{i,1} +   f_{i,2}$ with $ f_{i,1}\in  L^2(]0,T[; L^2(\R^2))$ and $ f_{i,2}\in  L^2(]0,T[; L^{4/3}(\R^2))$ it holds 
\EQ{
&\Bigl|\Pi_a \Bigl(P_{N_1}u_1,  P_{N_2} u_2 , P_{N_3} u_3 \Bigr)  \Bigr| \\
\lesssim& \Bigl(    N_3^{-1/2} \|P_{N_3} u_3 \|_{L^\infty_T L^2_x} +   \Bigl(\frac{N_3}{N_1}\Bigr)^{\frac{1}{2}}\cro{N_3}^{0+} 
\sup_{I\subset ]0,T[ \atop |I|\lesssim N_1^{-1}} \Bigl[ \|P_{N_3} f_{3,1} \|_{L^2_I L^2_{xy}} +  \|P_{N_3} f_{3,2} \|_{L^{\infty-}_I L^{1+}_{xy}}\Bigr]\Bigr)\\
&\times \prod_{i=1}^2 \Bigl( T^{-\frac{1}{2}}\|P_{N_i} u_i\|_{L_T^2 L^2_{xy}}+ \Bigl(\frac{N_3}{N_1}\Bigr)^{\frac{1}{2}}
\| P_{N_i} u_i\|_{L_T^\infty  L^2_{xy}}\\
&\qquad\qquad +T^\frac{1}{2}
\|P_{N_i} f_{i,1} \|_{L_T^2 L^2_{xy}}+ T^\frac{1}{4}N_1^{\frac{1}{4}+} \|P_{N_i} f_{i,2} \|_{L_T^2 L^{4/3}_{xy}}\Bigr) \,. \label{tritri2}
}
\end{proposition}
\begin{proof}
We used that according to \eqref{hom} with $ q=4+ $ and $ r=4+ $ and interpolation  it holds
$$
\|f\|_{L^{4}_T L^4_{xy}}\lesssim \|J_x^{0+} f\|_{X^{0,\frac{1}{2}-}_\pm} , \quad \forall f\in C^\infty_0 (\R^3)\; .
$$
Hence for $ N\lesssim N_1$, 
\EQN{
\|\eta_j P_{N}  f\|_{X^{0,-1/2,1}_\pm} \lesssim& N^{0+} \|  \eta_j P_{N}  f\|_{L^{\frac{4}{3}}_t L^\frac{4}{3}_{xy}}\\
\lesssim&  (TN_1^{-1})^\frac{1}{4} N^{0+}   \|\eta_j P_{N} f \|_{L^2_t L^{4/3}_{xy}}, \quad \forall  f\in L^2_t L^{4/3}_{xy}\; .
}
Therefore we may replace $\displaystyle \prod_{i=1}^2\Bigl(\sum_{j\in J_{N_1}}  \| P_{N_i} \eta_j  f_i \|_{L^2_I L^2_{xy}}^2 \Bigr)^{1/2}$ in the estimate on $ C_2 $ in the proof of Proposition \ref{prop_tri} by 
$$
\prod_{i=1}^2 \Bigl[ \Bigl(\sum_{j=1}^{N_1-1} \| \eta_j  P_{N_i}  f_{i,1} \|_{L^2_t L^2_{xy}}^2 \Bigr)^{1/2}+T^{-\frac{1}{4}} N_1^{\frac{1}{4}+}
 \Bigl(\sum_{j=1}^{N_1-1}  \|  \eta_j  P_{N_i} f_{i,2} \|_{L^{2}_t L^{4/3}_{xy}}^2 \Bigr)^{1/2}\Bigr]
$$
$$
\lesssim  \displaystyle \prod_{i=1}^2  \Bigl(\| P_{N_2}   f_{i,1} \|_{L^2_T L^2_{xy}}+
T^{-\frac{1}{4}}N_1^{\frac{1}{4}+}  \|  P_{N_2} f_{i,2} \|_{L^2_T L^\frac{4}{3}_{xy}}\Bigr)
$$
and $C_3 $ can be handled  in the same way. This completes the proof since $ f_1 $ and $ f_2$ are not involved in the estimates on $ C_1$ and $ A$ in the proof of Proposition \ref{prop_tri}.
\end{proof} 

The following proposition proves that the flow-map is Lipschitz in $ \overline{H}^{s-1,0} (\R^2)$. Even if no frequency envelop is needed in this step, we will need the same type of estimates but with a frequency envelop to treat the contribution of $ \partial_x^{-1} u_y $ for the LWP in the energy space. This is the reason why we state this proposition in the  frequency envelop space $ \overline{H}^{s-1,0}_{\om}$.
\begin{proposition}\label{prodif}
 Let $ s>\frac{3}{4} $
 and  $ u_1, u_2\in L^\infty_T H^{s,0}_\om $ be solutions of the KP-I equation then, setting $ w=u_1-u_2 $, it holds 
\EQ{
\|w\|_{L^\infty_T  \overline{H}^{s-1,0}_{\om}}^2
 \lesssim &\|w(0)\|_{ \overline{H}^{s-1,0}_{\om}}^2 + \|w\|_{L^\infty_T \overline{H}^{s-1,0}_{\om}}^2\Bigl(\sum_{i=1}^2(1+  \|u_i\|_{L^\infty_T H^{\frac{3}{4}+,0}_{xy}})(1+
  \|u_i\|_{L^2_T L^\infty_{xy}})  \|u_i\|_{ H^{s,0}_{\om}}\Bigr)\nonumber \\
 &\qquad\qquad \times \Bigl[ 1+ \sum_{i=1}^2 \Bigl( (1+\|u_i\|_{L^2_T L^\infty_{xy}})^2 \|u_i\|_{L^\infty_T H^{\frac{3}{4}+,0}}^2 \\
 &\qquad\qquad \qquad\qquad +\| J_x^{\max(0,s-1)+} u_i \|_{L^{2}_T L^\infty_{xy}}^2 + \|D_x^{s-(\frac{1}{2}+)} u_i\|_{L^4_{T}L^4_{xy}}^2 \Bigr)\Bigr]  .
}

\end{proposition}
 \begin{proof}
 We set  $z=u_1+u_2 $. To bound the low $x$-frequency part of $ w$ we apply \eqref{non-hom} with $(q_1,r_1)=(\infty,2) $ and $(q_2,r_2)=(2+,\infty) $  on the Duhamel formulation to  get for $ 0<N\lesssim 1$, 
 $$
 \|P_{N} w \|_{L^\infty_T \overline{H}^{s-1,0}}   \lesssim \|P_{N} w(0)\|_{\overline{H}^{s-1,0}}+
 N^{-\frac{3}{4}} N \|J^{0+}_x  P_{\lesssim 1} (zw) \|_{L^{2-}_T L^1_{xy}} \; .
 $$
Then writting 
  \begin{equation}\label{decompzw}
  P_{\lesssim 1} (z w)= P_{\lesssim 1}\Bigl(  \tilde{P}_{\lesssim 1}  z P_{\ll 1} w + P_{\ll 1} z \tilde{P}_{\lesssim 1}  w +
 \sum_{N_1\gtrsim 1} P_{N_1} z \tilde{P}_{N_1} w\Bigr)   \; ,
 \end{equation}
  we  observe that 
 \begin{align*}
 \|J^{0+}_x  P_{\lesssim 1} (zw) \|_{L^{2-}_T L^1_{xy}} & \lesssim \Bigl\| P_{\lesssim 1}\Bigl(  \tilde{P}_{\lesssim 1}  z P_{\ll 1} w + P_{\ll 1} z \tilde{P}_{\lesssim 1}  w +
 \sum_{N_1\gtrsim 1} P_{N_1} z \tilde{P}_{N_1} w\Bigr)\Bigr\|_{L^{2-}_T L^1_{xy}}\\
  & \lesssim T^{1/2} \Bigl(\|  \tilde{P}_{\lesssim 1}  z \|_{L^\infty_T L^2_{xy}} 
\|  \tilde{P}_{\lesssim 1} w \|_{L^\infty_T L^2_{xy}}  \\
&\qquad\qquad +\sum_{N_1\gtrsim 1} \|P_{N_1} z\|_{L^\infty_T L^2_{xy}}  \|\tilde{P}_{N_1} w\|_{L^\infty_T L^2_{xy}}\Bigr) \\
 &  \lesssim T^{1/2} \|z\|_{L^\infty_T H^{\frac{1}{4}+,0}} \|w\|_{L^\infty_T H^{-\frac{1}{4},0}}
 \end{align*}
 that gives, by summing over $ 0<N\lesssim 1 $, 
 \begin{equation}
 \|P_{\lesssim 1} w\|_{L^\infty_T \overline{H}^{s-1,0}}   \lesssim  \| w(0)\|_{\overline{H}^{s-1,0}}+
  T^{1/2} \|z\|_{L^\infty_T H^{\frac{1}{4}+,0}} \|w\|_{L^\infty_T H^{-\frac{1}{4},0}} \label{lowf} \; .
 \end{equation}
 Now for the high $x$-frequency part we make an energy estimate. We have to estimate 
 \begin{equation}
 I=\sum_{N\gg 1}  \langle N  \rangle^{2s-2}   \om_N^2  \Bigl| \int_0^t  \int_{\R^2} \partial_x P_N(z w ) P_N w \Bigr| 
 \end{equation}
 where we may decompose  $P_N(z w) $ as 
\begin{equation}\label{decom}
  P_N(z w)= P_N\Bigl(  \tilde{P}_N z P_{\ll N} w + P_{\ll N} z \tilde{P}_N w +
 \sum_{N_1\gtrsim N} \tilde{P}_{N_1} z P_{N_1} w\Bigr) . 
\end{equation}
 $\bullet $ Contribution of  $\sum_{0<N_3\ll N} \tilde{P}_{N}  w P_{N_3} z $.\\
 We denote by $ A$ this contribution. After integrating by parts, as in \eqref{intep},  we obtain
 $$
A =   \sum_{N \gg  1}\sum_{0<N_3\ll N} \om_N^2 \langle  N\rangle^{2s-2} N_3 \int_0^t \int_{\R^2} \Lambda_{a_1}(P_{N_3}z, \tilde{P}_N w) \cdot P_N w  \,dxdy ds
$$
with  $ a_1$ defined as in \eqref{a1} and thus satifying \eqref{estcom} and $\|a_1\|_{L^\infty}\lesssim 1 $.
 We may thus apply  Proposition \ref{propro2}   with $ f_{i,1} =  \tilde{P}_N z P_{\ll N} w + P_{\ll N} z \tilde{P}_N w
 $ and $ f_{i,2}= \sum_{N_1\gtrsim N} \tilde{P}_{N_1} z P_{N_1}w $ for $i=1,2$, and $ f_{3,1} =2 \sum_{i=1}^2 P_{\ll N_3} u_i P_{\sim N_3} u_i 
 $ and $ f_{3,2}= \sum_{i=1}^2 \sum_{N_4\gtrsim N_3} (P_{N_4} u_i P_{\sim N_4} u_i)  $ 
  to  get 
 \begin{align*}
|A|&\lesssim  \sum_{N\gg  1}  \sum_{0<N_3\ll N} \om_N^2 \langle N  \rangle^{2s-2} N_3
 \Bigl|  \int_0^t  \int_{\R^2}  \Lambda_a(P_{N_3} z,  \tilde{P}_{N} w) P_N w \Bigr| \nonumber\\
& \lesssim  \sum_{N\gg  1}  \sum_{0<N_3\ll N} \om_N^2 \langle N  \rangle^{2s-2}   \Bigl( T^{-1}\|P_{N} w\|_{L^2_{Txy}}^2\nonumber \\
& \quad+ \underbrace{\|P_{\ll N} z P_{\sim N} w  \|_{L^{2}_T L^2_{xy}}^2 +\|P_{\ll N} w P_{\sim N} z  \|_{L^{2}_T L^2_{xy}}^2 
  + (N^{\frac{1}{4}+} \sum_{N_2\gtrsim N}  \| P_{N_2} z P_{\sim N_2}w\|_{L^2_T L^\frac{4}{3}_{xy}})^2}_{A_1}\Bigr)\nonumber \\
  & \quad \times  \Bigl[ N_3^{{1\over 2}}   \|P_{N_3} z \|_{L^\infty_T L^2_x}  
   \\
   &\quad +\underbrace{\Bigl(\frac{N_3}{N}\Bigr)^{\frac{1}{2}}N_3^{\frac{1}{2}+}
\sum_{i=1}^2   \sup_{I\subset ]0,T[ \atop |I|\lesssim N^{-1}}  \Bigl( \|P_{\ll N_3} u_i P_{\sim N_3} u_i   \|_{L^{2}_I L^2_{xy}} 
  + \sum_{N_4\gtrsim N_3} \| P_{N_4} u_i P_{\sim N_4}u_i \|_{L^{\infty-}_I L^{1+}_{xy}} \Bigr)}_{A_2}\Bigr].
\end{align*}
To  estimate the contribution of $ A_1$ we first notice that 
 \begin{align*}
 \sum_{N\gg 1} \om_N^2 N^{2s-2}  \|P_{\ll N} z P_{\sim N} w  \|_{L^{2}_T L^2_{xy}}^2  & =\int_0^T  \sum_{N\gg 1} \om_N^2 N^{2s-2} \|P_{\ll N} z P_{\sim N} w  \|_{L^2_{xy}}^2 \\
 & \lesssim  \int_0^T \sum_{N\gg 1} \|z\|_{L^\infty_{xy}}^2 \| P_{\sim N} w\|_{H^{s-1,0}_\om}^2 \\
& \lesssim  \|z\|_{L^2_T L^\infty_{xy}}^2 \|w\|_{L^\infty_T H^{s-1,0}_\om}^2 \; .
 \end{align*}
 Then, for $ 3/4<s<1$, we notice that 
  \begin{align*}
 \sum_{N\gg 1} \om_N^2 N^{2s-2}\|P_{\ll N} w P_{\sim N} z  \|_{L^{2}_T L^2_{xy}}^2
  & \lesssim \sum_{N\gg 1}\|P_{\ll N} w\|_{L^\infty_T H^{s-1}}^2\| P_{\sim N} z  \|_{L^{2}_T L^\infty_{xy}}^2\\
  & \lesssim 
    \|D_x^{0+}z\|_{L^2_T L^\infty_{xy}}^2 \|  w\|_{L^\infty_T  H^{s-1,0}}^2 \; .
  \end{align*}
 whereas for $ s>1$ we may write 
   \begin{align*}
 \sum_{N\gg 1} \om_N^2 N^{2s-2}\|P_{\ll N} w P_{\sim N} z  \|_{L^{2}_T L^2_{xy}}^2 & \lesssim \sum_{N\gg 1}\|P_{\ll N} w\|_{L^\infty_T L^2_{xy}}^2\,  \om_N^2 \| D^{s-1}_x P_{\sim N} z  \|_{L^{2}_T L^\infty_{xy}}^2 \\
 & \lesssim 
    \|D_x^{(s-1)+}z\|_{L^2_T L^\infty_{xy}}^2 \|  w\|_{L^\infty_T  H^{s-1,0}}^2 \; .
   \end{align*}
  Finally,  for $ s>3/4 $, we can always assume that  $  \om_N^2 \lesssim  N^{0+} $, it  holds
  \begin{align}
   &N_1^{s-1}    \om_N^2 N^{\frac{1}{4}+} \sum_{N_1\gtrsim N}  \| P_{N_1} z P_{\sim N_1}w\|_{L^2_T L^\frac{4}{3}_{xy}} \\
   & \lesssim
  T^{1/2}  N_1^{s-1}  N^{\frac{1}{4}+} \sum_{N_1\gtrsim N} N_1^{-s+\frac{1}{2}+}\|D_x^{s-(\frac{1}{2}+)}P_{N_1} z\|_{L^4_T L^4_{xy}}N_1^{1-s}
  \|P_{\sim N_1}w\|_{L^\infty_T H^{s-1,0}}\nonumber \\
  & \lesssim T^{1/2}  N^{s-1}  N^{\frac{1}{4}+} N^{-2s+\frac{3}{2}+} \|D_x^{s-(\frac{1}{2}+)} z\|_{L^4_{T}L^4_{xy}} \|w\|_{L^\infty_T H^{s-1,0}}\nonumber \\
  & \lesssim T^{1/2} N^{-s+\frac{3}{4}+} \|D_x^{s-(\frac{1}{2}+)} z\|_{L^4_{T}L^4_{xy}} \|w\|_{L^\infty_T H^{s-1,0}}
  \end{align}
  that is acceptable.

 Estimating the contribution of $ A_2 $  exaclty as the similar term in the estimate of $ C $ in \eqref{DefC} we thus get 
  \begin{align}
|A|&\lesssim \Bigl[ \sum_{i=1}^2   \Bigl(1+  \|u_i\|_{L^\infty_T H^{\frac{3}{4}+,0}_{xy}}(1+
  \|u_i\|_{L^2_T L^\infty_{xy}}) \Bigr) \|u_i\|_{L^\infty_T H^{\frac{3}{4}+,0}_{xy}}\Bigr]\nonumber\\
& \hspace*{10mm}\times \Bigl( 1+\| J_x^{\max(0,s-1)+} z \|_{L^{2}_T L^\infty_{xy}}^2 + \|D_x^{s-(\frac{1}{2}+)} z\|_{L^4_{T}L^4_{xy}}^2 \Bigr)
 \|w\|_{L^\infty_T H^{s-1,0}_\om}^2.
  \label{eqr1}
 \end{align}

 $\bullet $ Contribution of   $\sum_{N_1\gtrsim N} P_{N_1} z \tilde{P}_{N_1} w$.\\
  We denote by $ B$ this contribution. In this case we have $ 1\ll N \lesssim N_1$. 
  We use the decomposition \eqref{decom} 
  and apply Proposition \ref{propro2} with $u_1=z$, $ u_2=w$,  $ f_{2,1} = P_{\ll N_1} z P_{\sim N_1} w + P_{\sim N_1} z P_{\ll N_1} w$ and $ f_{2,2}=  \sum_{N_2\gtrsim N_1} P_{N_2} z P_{\sim N_2} w$  , $ u_3 = w$, $ N_3=N$, $ f_{3,1} = P_{\ll N_3} z P_{\sim N_3} w + P_{\sim N_3} z P_{\ll N_3} w$ and $ f_{3,2}=  \sum_{N_4\gtrsim N_3} P_{N_4} z P_{\sim N_4} w$  to get  
  \begin{align*}
|B|&\lesssim  \sum_{N\gg  1} \sum_{N_1\gtrsim N}  \om_N^2 N  \langle N  \rangle^{2s-2}  \Bigl|  \int_0^t  \int_{\R^2}  P_{N_1} z \tilde{P}_{N_1} wP_N w \Bigr| \nonumber\\
& \lesssim  \sum_{N\gg  1} \sum_{N_1\gtrsim N} \Bigl(\frac{N}{N_1}\Bigr)^{2s-1} \om_{N_1} N_1^{s}\Bigl( T^{-\frac{1}{2}}\|P_{N_1} z\|_{L^2_{Txy}}+\|P_{N_1} (u_1^2+u_2^2)\|_{L^2_{Txy}}\Bigr)\nonumber \\
&   \times  \om_{N_1}N_1^{s-1}\Bigl( T^{-\frac{1}{2}}\|P_{N_1} w\|_{L^2_{Txy}}\\
&\quad + \underbrace{\|P_{\ll N_1} z P_{\sim N_1} w  \|_{L^{2}_T L^2_{xy}} +\|P_{\ll N_1} w P_{\sim N_1} z  \|_{L^{2}_T L^2_{xy}} 
  + N_1^{\frac{1}{4}+} \sum_{N_2\gtrsim N_1}  \| P_{N_2} z P_{\sim N_2}w\|_{L^2_T L^\frac{4}{3}_{xy}}}_{B_1}\Bigr) \nonumber \\
  & \hspace*{1mm} \times  \Bigl( \cro{N}^{-{1\over 2}}   \|P_{N} w\|_{L^\infty_T L^2_x}  +    \Bigl( \frac{N}{N_1}\Bigr)^{\frac{1}{2}} N^{0+}\nonumber \\
  &    \hspace*{10mm}\\
  &\times
  \sup_{I\subset ]0,T[ \atop |I|\lesssim N_1^{-1}}  \Bigl(\underbrace{ \|P_{\ll N} z P_{\sim N} w  \|_{L^{2}_I L^2_{xy}} +\|P_{\ll N} w P_{\sim N} z  \|_{L^{2}_I L^2_{xy}} 
  + \sum_{N_4\gtrsim N} \| P_{N_4} z P_{\sim N_4}w\|_{L^{\infty-}_I L^{1+}_{xy}}}_{B_2}\Bigr)\Bigr).\\
  \end{align*}
  The contribution of $ B_1$ can be estimated exactly as the one of $ A_1 $ just above. 
  Let us now estimate one by one  the contribution of  three terms appearing in $B_2$. First, proceeding as in \eqref{31},  it is direct to see that
     $$
  \sum_{N_4\gtrsim N} N^{0+} \| P_{N_4} z P_{\sim N_4}w\|_{L^{\infty-}_I L^{1+}_{xy}}  \lesssim 
  \|w\|_{L^\infty_I H^{-\frac{1}{4},0}} ( 1+\|z\|_{L^\infty_I H^{\frac{1}{4}+}}(1+ \|z\|_{L^2_I L^\infty_{xy}})) \; .
  $$
  Then, according to Lemma \ref{lemnew} and \eqref{der} we have 
  \begin{align*}
 N^{0+} \|P_{\ll N} z P_{\sim N} w  \|_{L^{2}_I L^2_{xy}} & \lesssim  N^{0+}\|P_{\ll N} z \|_{L^2_I L^\infty_{xy}} \|P_{\sim N} w  \|_{L^{\infty}_I L^2_{xy}}\\ 
  & \lesssim   N^{-\frac{1}{4}+}  \Bigl( \|J_x^{\frac{3}{4}} z \|_{L^\infty_I L^2_{xy}} +
  \|J_x^{\frac{3}{4}}  (u_1^2+u_2^2) \|_{L^{2}_I L^2_{xy} } \Bigr)  \|P_{\sim N} w  \|_{L^{\infty}_T L^2_{xy}}\\
   & \lesssim \|w\|_{L^\infty_T H^{-\frac{1}{4}+,0}} \sum_{i=1}^2 (1+\|u_i\|_{L^2_T L^\infty_{xy}}) \|u_i\|_{L^\infty_T H^{\frac{3}{4}+,0}}
  \end{align*} 
  and in the same way
   \begin{align*}
  N^{0+}   \|P_{\sim  N} z P_{\ll N} w  \|_{L^{2}_I L^2_{xy}} & \lesssim  N^{0+}  \|P_{\sim N} z \|_{L^2_I L^\infty_{xy}} \|P_{\ll N} w  \|_{L^{\infty}_I L^2_{xy}}\\ 
  & \lesssim   N^{-\frac{1}{4}+}  \Bigl( \|J_x^{\frac{3}{4}} z \|_{L^\infty_I L^2_{xy}} +
  \|J_x^{\frac{3}{4}}  (u_1^2+u_2^2) \|_{L^{2}_I L^2_{xy} } \Bigr)  \|P_{\ll N} w  \|_{L^{\infty}_T L^2_{xy}}\\
   & \lesssim \|w\|_{L^\infty_T H^{-\frac{1}{4}+,0}} \sum_{i=1}^2 (1+\|u_i\|_{L^2_T L^\infty_{xy}}) \|u_i\|_{L^\infty_T H^{\frac{3}{4}+,0}}.
\end{align*} 
Gathering all these estimates and making use of \eqref{der}  we eventually obtain   
  \begin{align}
  |B|& \lesssim \Bigl( \sum_{i=1}^2   (1+
  \|u_i\|_{L^2_T L^\infty_{xy}})  \|u_i\|_{L^\infty_T H^{s,0}_{\om}}\Bigr) \nonumber \\
  & \hspace*{10mm} \times  \|w\|_{L^\infty_T H^{s-1,0}_\om} \Bigl( 1+\| J_x^{\max(0,s-1)+}z \|_{L^{2}_T L^\infty_{xy}} + \|D_x^{s-(\frac{1}{2}+)} z\|_{L^4_{T}L^4_{xy}} \Bigr)\nonumber \\
   & \hspace*{10mm} \times \|w\|_{L^\infty_T H^{-\frac{1}{4}+,0}}\Bigl( 1+ \sum_{i=1}^2 (1+\|u_i\|_{L^2_T L^\infty_{xy}}) \|u_i\|_{L^\infty_T H^{\frac{3}{4}+,0}}\Bigr) .
\label{eqr2}
 \end{align}
  $\bullet $ Contribution of   $\sum_{0<N_3 \ll N} \tilde{P}_{N} z P_{N_3} w$. \\
    We denote by $ C$ this contribution. The treatment is very similar to the preceding one.   We  apply Proposition \ref{propro2} with $u_1=z$, $ u_2=w$,  $ f_{2,1} = P_{\ll N} z P_{\sim N} w + P_{\sim N} z P_{\ll N} w$ and $ f_{2,2}=  \sum_{N_1\gtrsim N} P_{N_1} z P_{\sim N_1} w$  , $ u_3 = w$, $ N_3=N$, $ f_{3,1} = P_{\ll N_3} z P_{\sim N_3} w + P_{\sim N_3} z P_{\ll N_3} w$ and $ f_{3,2}=  \sum_{N_4\gtrsim N_3} P_{N_4} z P_{\sim N_4} w$  to get  
 \EQ{
|C| 
\lesssim& \sum_{N\gg  1}  \sum_{0<N_3\ll N}  \om_N^2 \langle N  \rangle^{2s-2}  N\Bigl|  \int_0^t  \int_{\R^2}  \tilde{P}_N z  P_{N_3} w  P_N w \Bigr| \\
 \lesssim&  \sum_{N\gg  1}  \sum_{0<N_3\ll N} \om_N N^{s}\Bigl( T^{-\frac{1}{2}}\|P_{N} z\|_{L^2_{Txy}}+\|P_{N} (u_1^2+u_2^2)\|_{L^2_{Txy}}\Bigr)\\
&   \times  \om_N N^{s-1}\Bigl( T^{-\frac{1}{2}}\|P_{N} w\|_{L^2_{Txy}}\\
&+ \underbrace{\|P_{\ll N} z P_{\sim N} w  \|_{L^{2}_T L^2_{xy}} +\|P_{\ll N} w P_{\sim N} z  \|_{L^{2}_T L^2_{xy}} 
  + N^\frac{1}{4} \sum_{N_1\gtrsim N}  \| P_{N_1} z P_{\sim N_1}w\|_{L^2_I L^\frac{4}{3}_{xy}}}_{C_1}\Bigr) \\
  &\times  \Bigl( \cro{N_3}^{-{1\over 2}}   \|P_{N_3} w\|_{L^\infty_T L^2_x}  +( \frac{N_3}{N})^{\frac{1}{2}} N^{0+}_3\\
  &\times  
  \sup_{I\subset ]0,T[ \atop |I|\lesssim N_1^{-1}}  \Bigl(\underbrace{ \|P_{\ll N_3} z P_{\sim N_3} w  \|_{L^{2}_I L^2_{xy}} +\|P_{\ll N_3} w P_{\sim N_3} z  \|_{L^{2}_I L^2_{xy}} 
  + \sum_{N_4\gtrsim N_3}  \| P_{N_4} z P_{\sim N_4}w\|_{L^\infty_I L^1_{xy}}}_{C_2}\Bigr) \Bigr)\quad . \label{eqC}
}
Clearly, the contributions of $C_1$ and $ C_2$ can be estimated exactly as $ A_1 $ and $ B_2 $.  Using that 
$ \displaystyle\sum_{0<N_3\ll N}  N_3^{-{1\over 2}}   \|P_{N_3} w\|_{L^\infty_T L^2_x} \lesssim \|w\|_{L^\infty_T \overline{H}^{-\frac{1}{4},0}} $ we eventually obtain 
   \begin{align}
  |C|& \lesssim \Bigl( \sum_{i=1}^2   (1+
  \|u_i\|_{L^2_T L^\infty_{xy}})  \|u_i\|_{L^\infty_T H^{s,0}_{\om}}\Bigr) \nonumber \\
  & \hspace*{10mm} \times  \|w\|_{L^\infty_T H^{s-1,0}_\om} \Bigl( 1+\| J_x^{\max(0,s-1)+}z \|_{L^{2}_T L^\infty_{xy}} + \|D_x^{s-(\frac{1}{2}+)} z\|_{L^4_{T}L^4_{xy}} \Bigr)\nonumber \\
   & \hspace*{10mm} \times \|w\|_{L^\infty_T \overline{H}^{-\frac{1}{4}+,0}}\Bigl( 1+ \sum_{i=1}^2 (1+\|u_i\|_{L^2_T L^\infty_{xy}}) \|u_i\|_{L^\infty_T H^{\frac{3}{4}+,0}}\Bigr) .
\label{eqr3}
 \end{align}
Gathering \eqref{eqr1}, \eqref{eqr2} and \eqref{eqr3}, the desired estimate follows.
 \end{proof}
 \section{Unconditional local well-posedness in $ H^{s,0} $, $ s>3/4$.}\label{sect7}
We will use the classical dilation argument since our approach does not enable us  to keep a power of $ T $ in our a priori estimates.   Recall that $u\in L^\infty(0,T;H^{s,0})$ is a solution to KP  on $ ]0,T[ $ if and only if, for some $ 0<\lambda\le 1$,  $(t,x,y)\mapsto u_\lambda(t,x,y):=\lambda^{2} u(\lambda^{3}t, \lambda x, \lambda^2 y)\in L^\infty(0, \lambda^{-3} T;H^s)  $ is  a  solution to KP on $ ]0,\lambda^{3} T[ $ and that 
\begin{equation}\label{rescale}
\|u_{\lambda}(t)\|_{H^{s,0}} \lesssim\  \lambda^{1/2} \|u(\lambda^3 t)\|_{H^{s,0}} , \quad \forall 0\le t\le  \lambda^{-3} T \; .\end{equation}
\subsection{Unconditional uniqueness}
Let $ s>3/4$ and let $u_1 $ and $ u_2 $ be two solutions of the KP equation associated to the same initial data $\varphi\in H^{s,0}(\R^2) $ that belong to $ L^\infty(0,T;H^{s,0}(\R^2)) $ for some $ T>0 $. For any given $ 0<\varepsilon<1 $, the above dilation argument ensures that  taking $ \lambda=[\varepsilon^{-1}(1+\max_{i=1,2} \|u_i\|_{L^\infty_T H^{s,0}})]^{-2}$, the associated dilated solutions $u_{i,\lambda} $ satisfy  $ \|u_{i,\lambda}\|_{L^\infty(0,T_\varepsilon;H^{s,0})}\lesssim \varepsilon$, $i=1,2$,
 with 
 $$ T_\varepsilon \sim [\varepsilon^{-1} (1+\max_{i=1,2} \|u_i\|_{L^\infty_T H^{s,0}})]^{6}T  \; .
 $$
 In particular $ T_\varepsilon> 1 $ as soon as $ \varepsilon\ll T^{1/6} $. 
 Noticing that $ (u_1-u_2)(0)=0\in H^{s-1}_{\frac{3}{4}} $, Proposition \ref{prodif}  together with Proposition \ref{pro1}  ensure that, taking $ 0< \varepsilon \ll  T^{1/6} \wedge 1$ small enough, it holds  $ u_{1,\lambda}\equiv u_{2,\lambda} $ on $ [0,1] $. Coming back to the solutions $ u_1$ and $ u_2 $, this forces $ u_1\equiv u_2 $ on $[0, ( \varepsilon^{-1} (1+\max_{i=1,2} \|u_i\|_{L^\infty_T H^{s,0}}))^{-6}] $ and the uniqueness on $ [0,T] $ follows by iterating this argument a finite number of times.
 \subsection{Local Existence and strong continuity}\label{sub72}
 It follows from the  dilation argument that the LWP of KP  in $ H^{s,0}(\R^2) $ for small $H^{\frac{3}{4}+,0}$-initial data with a
 time of existence  $ T\ge 1 $  ensures the LWP
  of KP   for arbitrary large $H^{s,0}$-initial data  $u_0$ with a maximal time of existence
  $$T\gtrsim  (1+\|u_0\|_{\frac{3}{4}+})^{-6}\, .$$
  We can thus restrict ourselves to prove a small initial data LWP result.
   First let  $ u_0\in H^{\infty ,0}(\R^2) $ with $ \|u_0\|_{H^{\frac{3}{4}+,0}}\ll 1 $.  According to  \cite{MoSaTz4} for the KP-I equation and \cite{Bourgain3} for the KP-II equation, there exists a solution $ u\in C([0,T]; H^{\infty,0}(\R^2)) $ of the KP equation emanating from  the initial data $ u_0$ with $ T=T(\|u_0\|_{H^{\frac{3}{2}+,0}}) $. 
 Taking $ s=\frac{3}{4}+$ and $ \omega\equiv 1 $  in \eqref{estim}  it follows  from a direct continuity argument  that there exists $ \varepsilon_0>0 $ such that  if $ \|u_0\|_{H^{\frac{3}{4}+,0}}< \varepsilon_0 $ then the solution $ u$ satisfies $\|u(t)\|_{H^{\frac{3}{4}+,0}} < 2 \varepsilon_0  $  on $[0, \min(T,1)] $. Using again \eqref{estim}  with $ \omega\equiv 1$ we deduce  that $\|u(t)\|_{H^{s,0}} \lesssim \|u_0\|_{H^{s,0}}  $  on $[0, \min(T,1)] $ for $ s>3/4$. The above LWP result then ensures that $ u\in C([0,1];H^{\infty,0}(\R^2)) $ with $ \|u\|_{L^\infty(]0,1[;H^{s,0})} \lesssim \|u_0\|_{H^{s,0}} $ for any $ s>3/4$. 
 
   Now let us fix $u_0\in H^{s,0}(\R^2) $ with $  s>3/4$ and  $ \|u_0\|_{H^{\frac{3}{4}+,0}}< \varepsilon_0 $ .  Setting $ u_{0,n}=P_{\le n} u_0 $, it is clear that the sequence
  $ \{u_{0,n}\}_{n\ge 1} $ is included in $ H^{\infty,0} (\R^2)$ and converges to
$ u_0 $ in $H^{s,0}(\R^2) $. We deduce from  above that the emanating sequence of solutions $\{u_{n}\}_{n\ge 1} $ to the KP  equation is  including in  $  C([0,1];H^{\infty,0}(\R)) $ and satisfies  $ \|u_n\|_{L_T^\infty  H^{s}}^2\le 2  \|u_0\|_{H^{s}}^2$.
   Moreover, since by construction $P_{\le 1} (u_{0,n_1}-u_{0,n_2}) \equiv 0 $ for any $n_1,n_2\ge 1$,  we infer from Propositon \ref{prodif}   that  
\begin{equation}\label{lip}
\|u_{n_1}-u_{n_2} \|_{L^\infty_1 H^{-\frac{1}{4}+,0}_\frac{3}{4}} \lesssim \|u_{0,n_1} -u_{0,n_2}\|_{H^{-\frac{1}{4}+,0}_{\frac{3}{4}}}
\sim  \|u_{0,n_1} -u_{0,n_2}\|_{H^{-\frac{1}{4}+,0}} , \quad n_1, n_2 \ge 1 \; .
\end{equation}
   
   Therefore  $\{u_{n}\}_{n\ge 1} $ is a Cauchy sequence in $ C([0,1]; H^{-\frac{1}{4}+,0}(\R^2)) $ and thus converges to some function $u\in L^\infty(]0,1[;H^{s,0} (\R^2) ) $ strongly in  $ C([0,1]; H^{s',0} (\R^2) ) $ for any $ s'<s$ and in particular in $C([0,1]; L^2(\R^2))) $.We can thus pass to the limit on the nonlinear term $ u_n^2 $ and $ u $ satisfies  \eqref{eq:kpI} at least in the distributional sense.
   Now, coming back to  $ \{u_{0,n}\}_{n\ge 1} $,  we recall that, following \cite{KoTz},  for any sequence $ \{u_{0,k}\}_{k\ge 1} $ converging to $u_0$ in $
 H^{s,0}(\R^2) $
there exists a  dyadic sequence $\{\om_N\}$  satisfying Definition \ref{deftenv} for some $ 1<\delta<1+ $ 
 such that
\begin{equation}\label{env}
\|u_0\|_{H^{s,0}_\om} <\infty , \quad  \sup_{k\ge 1} \|u_{0,k} \|_{H^{s,0}_\om} <\infty \quad \text{and} \quad  \om_N \nearrow +\infty \; .
\end{equation}
Applying again \eqref{estim} with this sequence $ \{\om_N \} $ we obtain that the emanating solutions satisfy 
$$
\sup_{n\ge 1} \|u_n \|_{L^\infty(0,1;H^s_\om)} \lesssim  \sup_{n\ge 1} \|u_{0,n} \|_{H^s_\om} <\infty \, .
$$
Combining this last estimate with \eqref{lip}, we deduce  that actually $\{u_{n}\}_{n\ge 1} $ is a Cauchy sequence in $ C([0,1]; H^{s,0}(\R^2)) $ and thus $u\in C([0,1]; H^{s,0}(\R^2)) $.
    \subsection{Continuity with respect to initial data}
We make use again of the frequency envelop $ \{\omega_N\}$. 
Let $ \{u_{0,k}\}_{k\ge 1}  \subset
H^{s,0}(\R^2)  $, with $\sup_{k\ge 1} \|u_{0,k}\|_{H^{s,0}} \le 2 \|u_0\|_{H^{s,0}}<\epsilon_0 $,  that converges to $ u_0 $ in $  H^{s,0}(\R^2) $ and let  $\{\om_N\}$ satisfying \eqref{env}. Applying \eqref{estim}   with $\{\om_N\}$ we infer that the sequence of solution $\{u_k\}_{k\ge 1} $ emanating from $ \{u_{0,k}\}_{k\ge 1} $ is bounded in
$ L^\infty_1 H^{s,0}_\om $. It thus suffices to prove the convergence of $ u_k $ towards $ u $ in $ L^\infty_1 H^{s',0}$ for some $s'\le s$. However, we can not apply directly \eqref{lip} here since $\{u_0-u_{0,k} \}_{k\ge 1} $ may not belong to $ H_{\frac{3}{4}}^{\frac{3}{4}+}$. Instead, we set $ u_0^n=P_{\le n} u_0 $ and  $ u_{0,k}^n =P_{\le n} u_{0,k} $ and we denote by respectively $u^n$ and $u_k^n $ the associated solutions in $C([0,1]; H^{s,0}(\R^2)$ constructed in Subsection \ref{sub72}. By the triangle inequality, for any $s'\le s $, we have 
$$
\|u-u_k\|_{L^\infty_1 H^{s',0}}\le \|u-u^n\|_{L^\infty_1 H^{s',0}}+\|u^n-u_k^n\|_{L^\infty_1 H^{s',0}}+\|u_k^n-u_k\|_{L^\infty_1 H^{s',0}}\; .
$$
Now, according to the Lipschitz bound \eqref{lip} it holds 
\begin{align}
    \lim_{n\to \infty} \sup_{k\ge 1}
    (\|u-u^n\|_{L^\infty_1 H^{s',0}}
    &+\|u-u^n_k\|_{L^\infty_1 H^{s',0}})\nonumber\\
     &\lesssim \lim_{n\to \infty} \sup_{k\ge 1}(
     \|P_{>n} u_0\|_{H^{s',0}}
     + \|P_{>n} u_{0,k}\|_{H^{s',0}})
     =0\, ,
\end{align}
with $ s'=-\frac{1}{4}+$. It thus remains to prove that for any fixed $ n\ge 1$, 
$$
\lim_{k\to \infty}\|u^n-u_k^n\|_{L^\infty_1 H^{s',0}}=0
$$
that  is a direct consequence of the LWP  of the KP equations in $ H^{2,0}(\R^2) $ (see \cite{MoSaTz4} and \cite{Bourgain3}) since $u_{0}^n$ and $ u_{0_k}^n $ belong to $H^{2,0}(\R^2) $ with $\displaystyle\lim_{n\to\infty}\|u_{0}^n- u_{0_k}^n\|_{H^{2,0}}= 0 $. This  ensures that $\{u_k\}_{k\ge 1}$  converges to $ u $ in $C([0,1];H^{s,0}(\R^2)) $ and proves the continuity with respect to initial data.

 \section{Global Well-posedness  in the energy space for the KP-I equation}\label{sect8}
 In this section we show how we can easily recover the global well-posedness result of the KP-I equation in the energy space.
\subsection{A priori estimate in the energy space}
We define the following anisotropic space constructed on the energy of the KP-I equation :
$$
E^s(\R^2) =\{\phi\in L^2(\R^2), \; \|\phi\|_{E^s}<\infty \} 
$$
with
\begin{equation}
\|\phi\|_{E^{s}}^2= \|\phi\|_{H^{s,0}}^2+ \|\partial^{-1}_x \phi_y\|_{H^{s-1,0}}^2 \; .
\end{equation}
\begin{proposition}\label{pro8}
Let  $0<T<1$, $s\ge 1 $ and  $u\in E^s $ be a solution of the KP-I equation then it holds 
\EQ{\label{equy}
&\|\partial^{-1}_x u_y \|_{L^\infty_T H^{s-1,0}_\om}^2 \\
\lesssim&  \|\partial^{-1}_x u_y(0)\|_{H^{s-1,0}_\om}^2+   \|\partial_x^{-1} u_y \|_{L^\infty_T H^{s-1,0}_\om}^2\Bigl((1+  \|u\|_{L^\infty_T H^{\frac{3}{4}+,0}_{xy}})(1+
  \|u\|_{L^2_T L^\infty_{xy}}) \Bigr) \|u\|_{L^\infty_T H^{s,0}_{\om}} \\
&\times \Bigl[ 1+  \Bigl( (1+\|u\|_{L^2_T L^\infty_{xy}})^2 \|u\|_{L^\infty_T H^{\frac{3}{4}+,0}}^2 +\| J_x^{\max(0,s-1)+} u \|_{L^{2}_T L^\infty_{xy}}^2 + \|D_x^{s-(\frac{1}{2}+)} u\|_{L^4_{T}L^4_{xy}}^2 \Bigr)\Bigr] \; .
}
\end{proposition} 
\begin{proof}
We notice that setting $ \vartheta=\partial_x^{-1} u_y $, $\vartheta$  satisfies the equation 
$$
(\vartheta_t + \vartheta_{xxx} +u \vartheta_x )_x - \vartheta_{yy} = 0 
$$
that is an equation  of the same type as the equation satisfied by the difference  $ w$ of two solutions.
Therefore  the estimates on $ \partial_x^{-1} u_y $ are   very close to the one on $ w$ in Proposition \ref{prodif}. 
To bound the low $x$-frequency part of $ \partial_x^{-1} u_y $ we apply directly the Strichartz estimates \eqref{non-hom} on the integral formulation and make use of \eqref{decompzw} to  get 
\EQN{
&\|P_{\lesssim 1} \partial_x^{-1} u_y \|_{L^\infty_T H^{s-1,0}} \\
\lesssim& \|P_{\lesssim 1} w(0)\|_{H^{s-1,0}}+\|J^{0+}_x  P_{\lesssim 1} (u  \partial_x(\partial_x^{-1} u_y) ) \|_{L^{2-}_T L^1_{xy}}\\
\lesssim& \|\partial_x^{-1} u_y(0)\|_{H^{s-1,0}} \\
&+\Bigl\| P_{\lesssim 1}\Bigl(  P_{\lesssim 1}  u \partial_x P_{\ll 1} \partial_x^{-1} u_y + P_{\ll 1} u \partial_x P_{\lesssim 1} \partial_x^{-1} u_y  +
 \sum_{N_1\gtrsim 1} P_{N_1} u \partial_x \tilde{P}_{N_1} \partial_x^{-1} u_y\Bigr)\Bigr\|_{L^{2-}_T L^1_{xy}}.
}
Direct considerations lead to 
\EQN{
&\|P_{\lesssim 1} u \partial_x P_{\ll 1} \partial_x^{-1} u_y\|_{L^{2-}_T L^1_{xy}}+\|P_{\ll 1}u \partial_x P_{\lesssim 1} \partial_x^{-1} u_y\|_{L^{2-}_T L^1_{xy}} \\
\lesssim& T^{1/2}  \|  P_{\lesssim 1}  u \|_{L^\infty_T L^2_{xy}} 
\|P_{\lesssim 1} \partial_x^{-1} u_y \|_{L^\infty_T L^2_{xy}} 
}
whereas by the H\"older, Bernstein and Cauchy-Schwartz inequalities, we observe that 
\EQN{
\Bigl\|  \sum_{N_1\gtrsim 1}  P_{N_1} u \partial_x \tilde{P}_{N_1} \partial_x^{-1} u_y\Bigr\|_{L^1_{xy}} \lesssim &
   \sum_{N_1\gtrsim 1} N_1 \| P_{N_1} u \|_{L^2_{xy}}\|\tilde{P}_{N_1} \partial_x^{-1} u_y\Bigr\|_{L^2_{xy}}\\
   \lesssim& \|u\|_{H^{1,0}} \| \partial_x^{-1} u_y\|_{L^2_{xy}}\; .
}
Gathering the above estimates we infer that 
\begin{align}
 \|P_{\lesssim 1} \partial_x^{-1}&u_y \|_{L^\infty_T H^{s-1,0}}     \lesssim \|\partial_x^{-1} u_y(0)\|_{H^{s-1,0}}+T^{1/2} 
  \|u\|_{L^\infty_T H^{1,0}} \| \partial_x^{-1} u_y \|_{L^\infty_T L^2_{xy}}.
\end{align}

Now for the high $x$-frequency part we make an energy estimate. We have to estimate 
 \begin{equation}
 \sum_{N\gg 1}  \om_N^2 \langle N  \rangle^{2s-2}  \Bigl| \int_0^t  \int_{\R^2}P_N(u \partial_x (\partial_x^{-1} u_y) ) P_N \partial_x^{-1} u_y\Bigr| 
 \end{equation}
where we may decompose  $P_N(u  \partial_x (\partial_x^{-1} u_y)) $ as 
\EQN{
  P_N(u \partial_x(\partial_x^{-1} u_y))  =& P_N\Bigl(  P_{\ll N} u   \partial_x
  \tilde{P}_N  (\partial_x^{-1} u_y )+ \tilde{P}_N u \partial_x P_{\ll N}  (\partial_x^{-1} u_y)\\
&\qquad +
 \sum_{N_1\gtrsim N} \tilde{P}_{N_1} u  \partial_x P_{N_1} (\partial_x^{-1} u_y)\Bigr)  \\
=& I_N + I\! I_N + I\! I \! I_N \; .
}
Note that the contribution of the last term above does not involve functions with $x$-frequencies less than $1$.  Moreover, it is worth noticing that,  by possibly integrating by parts for the first term, we will always be able to put a $x$-derivative on the functions  with $x$-frequencies less than $1$ when estimating the contributions of the first two terms. This remark is actually crucial in view of the weight $ N_3^{-3/4} $ that appears in the right-hand side of \eqref{tritri2}. 

The contribution of  $ I_N $ can be handled exactly as $ A$ in \eqref{eqr1} by replacing $ z$ by  $ u $   and $ w $ by $\partial_x^{-1} u_y$
to get 
\EQN{
&\sum_{N\gg 1} \om_N^2  \langle N  \rangle^{2s-2}  \Bigl| \int_0^t  \int_{\R^2} I_N  P_N \partial_x^{-1} u_y\Bigr| \\
\lesssim& \Bigl[   \Bigl(1+  \|u\|_{L^\infty_T H^{\frac{3}{4}+,0}_{xy}}(1+
\|u\|_{L^2_T L^\infty_{xy}}) \Bigr) \|u\|_{L^\infty_T H^{\frac{3}{4}+,0}_{xy}}\Bigr]\\
& \quad \times \Bigl( 1+\| J_x^{\max(0,s-1)+} u \|_{L^{2}_T L^\infty_{xy}}^2 + \|D_x^{s-(\frac{1}{2}+)} u\|_{L^4_{T}L^4_{xy}}^2\Bigr)
\|\partial_x^{-1} u_y \|_{L^\infty_T H^{s-1,0}_\om}^2.
}
In the same way, the contribution of  $ I\! I_N $ and $ I\! I \! I_N $  can be handled as  respectively $ C $ in \eqref{eqr1},  by  replacing $ \tilde{P}_N z, \; 
P_{N_3} w $  and $P_N w $ by $ \tilde{P}_N u, \;  \frac{N_3}{N} P_{N_3} \partial_x^{-1} u_y $ and $P_N  \partial_x^{-1} u_y$, and  $ B $ in \eqref{eqr2}  by  replacing $ z$ by  $ u$ and $ w$ by $\partial_x^{-1} u_y $ . We then eventually get 
\EQN{
&\sum_{N\gg 1}  \om_N^2 \langle N  \rangle^{2s-2}  \Bigl| \int_0^t  \int_{\R^2} (I\! I_N +  I\! I \! I_N) P_N \partial_x^{-1} u_y\Bigr| \\
\lesssim& \Bigl(   (1+\|u\|_{L^2_T L^\infty_{xy}})  \|u\|_{L^\infty_T H^{s,0}_{\om}}\Bigr) \|\partial_x^{-1}u_y  \|_{L^\infty_T H^{s-1,0}_\om}^2 \\
&  \times  \Bigl( 1+\| J_x^{\max(0,s-1)+}u \|_{L^{2}_T L^\infty_{xy}}^2 + \|D_x^{s-(\frac{1}{2}+)} u\|_{L^4_{T}L^4_{xy}}^2+ (1+\|u\|_{L^2_T L^\infty_{xy}}^2) \|u\|_{L^\infty_T H^{\frac{3}{4}+,0}}^2\Bigr) \; .
\label{eqr33}
}
Then \eqref{equy} follows by gathering the above estimates.
\end{proof} 

\subsection{Well-posedness in $ E^s(\R^2) $ for $s\ge 1$. } Let $ u_0\in E^s(\R^2)$ with $ s\ge 1 $. From the preceding section we know that by possibly rescaling $ u_0 $ there exists a solution $ u\in C([0,1]; H^{s,0}(\R^2)) $ to the KP equation emanating from $ u_0$  that  moreover satisfies 
$$
\|u\|_{L^\infty(0,1;H^{s,0})} \lesssim \|u_0\|_{H^{s,0}} \; .
$$
Setting again $u_{0,n} :=P_n u_0 $ with $ n\ge 1$, we know from the preceding section that  the emanating solution $ u_n \in C([0,1]; H^{\infty,0}(\R^2))$ tends to $ u $ in $C([0,1]; H^{s,0}(\R^2))$. Moreover by  classical energy method (see for instance [\cite{Kenig1}, Lemma 1.3]), one can prove that $u_n\in C(0,T;E^s(\R^2)) $ as soon as
$$
\|u_{n,x}\|_{L^1_T L^\infty_{xy}}  <\infty \; .
$$
Since according to \eqref{ImprovedStri2}, 
$$
\|u_{n,x}\|_{L^1_T L^\infty_{xy}} \lesssim \|u_n\|_{L^\infty_1 H^{2,0}}
$$
it follows that $ u_n\in C([0,1];E^s(\R^2)) $ and by possibly re-scaling again, \eqref{equy} with $\om_N\equiv 1$ leads to 
$$
\|\partial_x^{-1} u_{n,y} \|_{L^\infty_1 H^{s-1,0}} \lesssim \|\partial_x^{-1} \partial_y u_{0,n}\|_{H^{s-1,0}} 
\lesssim \|\partial_x^{-1} \partial_y u_{0}\|_{H^{s-1,0}}\; .
$$
This ensures that $u\in L^\infty(0,1; E^s(\R^2)) $. 

Now, we will need the following anisotropic Sobolev inequality (see \cite{BIN})
\begin{lemma}\label{sobolev}
For $2\leq p\leq 6$ there exists $C>0$ such that for every $u\in
H^{\infty}_{-1}(\R^2)$,
\begin{equation}\label{malak}
\|u\|_{L^p(\R^2)}\leq C\|u\|_{L^2(\R^2)}^{\frac{6-p}{2p}}\,\,
\|u_x\|_{L^2(\R^2)}^{\frac{p-2}{p}}\,\,
\|\partial_x^{-1}u_y\|_{L^2(\R^2)}^{\frac{p-2}{2p}}\,\, .
\end{equation}
\end{lemma}
This lemma ensures that $u_n \in C([0,1];L^3(\R^2)) $ and that $u_n\to u $ in $  C([0,1];L^3(\R^2))$.
 Moreover, by using for instance an exterior regularization by a sequence of smooth functions as in \cite{MoSaTz4} on can check that 
$$
E(u_n(t))= E(u_{0,n}) ,\; \forall t\in [0,1] \; 
$$
where $ E(\cdot) $ is the energy functional defined in \eqref{laws}.
Since $ \partial_x^{-1}u_{n,y}  \rightharpoonup  \partial_x^{-1} u_y $ in $L^2(\R^2)$, passing to the limit in $ n $ this ensures that 
$
 E(u(t))\le  E(u_{0}) ,\; \forall t\in [0,1] , 
$
and the reversibility with time of the equation ensures that actually 
$$
 E(u(t))= E(u_{0}) ,\; \forall t\in [0,1] \; .
$$
This forces $ t\mapsto \|\partial_x u_y(t)\|_{L^2_{xy}}^2 $ to be continuous and since clearly the map 
$ t\mapsto \partial_x^{-1} u_y $ belongs to $C_w([0,1];L^2(\R^2)) $, the strong continuity of this last  map with values in $ L^2(\R^2) $ follows. Note also that the conservation of the momentum and of the energy of the solution $ u $ 
together with \eqref{malak} ensure that $ u $ is bounded in $E^1(\R^2) $ on its interval of existence. This fact combined with the LWP in $ H^{s,0}(\R^2) $, $ s>3/4$, ensure that actually $u\in C(\R_+; H^{s,0}(\R^2))\cap L^\infty(\R_+; E^s(\R^2)) $ with 
 $\partial_x^{-1} u_y \in  C(\R_+; L^2(\R^2))$. 

Now by using an acceptable  frequency envelop such that 
$$
\|u_{0}\|_{H^{s,0}_\om}^2+\|\partial_x^{-1} u_{0,y}\|_{H^{s-1,0}_\om}^2<\infty\quad  \text{and} \quad \om_N\nearrow +\infty 
$$
\eqref{equy} this times leads, after a possible rescaling, to 
$$
\|\partial_x^{-1} u_{n,y} \|_{L^\infty_T H^{s-1,0}_\om} \lesssim \|\partial_x^{-1} \partial_y u_{0,n}\|_{H^{s-1,0}_\om} 
\lesssim \|\partial_x^{-1} \partial_y u_{0}\|_{H^{s-1,0}_\om}\; .
$$
with $ T>0 $ depending on $\|u_{0}\|_{H^{s,0}_\om}^2+\|\partial_x^{-1} u_{0,y}\|_{H^{s-1,0}_\om}^2$.
 Combining this last estimate with the fact that $\partial_x^{-1} u_y \in  C(\R_+; L^2(\R^2))$, we deduce that 
 $\partial_x^{-1} u_y \in  C([0,T]; H^{s,0}(\R^2))$ and the invariance of the equation with respect to temporal translations enables to conclude that $ u\in C(\R_+;E^s(\R^2)) $. 
 
 It remains to prove   the continuity with respect to initial data  of 
 the map $ u_0 \mapsto \partial_x^{-1} u_y $ from $E^s(\R^2) $ into $ C(\R_+;H^{s-1,0}(\R^2)) $. First  the continuity of the map $u_0\mapsto u $ from $H^{s,0}(\R^2) $ into $C([0,1];H^{s,0}(\R^2)) $ ensures  that for any $ \varphi\in L^2(\R^2) $, $ u_0 \mapsto ( \partial_x^{-1} u_y, \varphi)_{L^2(\R^2)}  $ is continuous from $E^s(\R^2) $ into $C([0,1]) $ and, making use of the conservation of the energy for $ E^s(\R^2) $-solutions together with  \eqref{malak}, that 
 $ u_0\mapsto \|\partial_x u_y\|_{L^2}^2 $  is continuous from $E^s(\R^2) $ into $C([0,1]) $. This ensures that $ u_0 \mapsto \partial_x^{-1} u_y $ is continuous from $E^s(\R^2) $ into $ C(\R_+;L^2(\R^2)) $. To upgrade to the continuity with values in $ C(\R_+;H^{s-1,0}(\R^2)) $ we proceed as above by using an adapted frequency envelop. 
  Let $ \{u_{0,k}\}_{k\ge 1}  \subset
E^s(\R^2)  $, with $\sup_{k\ge 1} \|u_{0,k}\|_{E^s} \le 2 \|u_0\|_{E^s} $,  that converges to $ u_0 $ in $ E^s(\R^2) $
and let  $\{\om_N\}$ satisfying
\EQN{
\|u_{0}\|_{H^{s,0}_\om}^2+\|\partial_x^{-1} u_{0,y}\|_{H^{s-1,0}_\om}^2<&\infty \\
\sup_{k\ge 1}( \|u_{0_k}\|_{H^{s,0}_\om}^2+\|\partial_x^{-1}\partial_y  u_{0,k}\|_{H^{s-1,0}_\om}^2)<&\infty \\
\text{and} \quad \om_N\nearrow +\infty 
}
\eqref{equy}  leads, after a possible rescaling, to 
$$
\|\partial_x^{-1} u_{k,y} \|_{L^\infty_T H^{s-1,0}_\om} \lesssim \|\partial_x^{-1} \partial_y u_{0,k}\|_{H^{s-1,0}_\om} 
<\infty
$$
for some $ T>0 $ and the desired continuity follows on $[0,T] $ and then on $\R_+ $  by invariance by temporal translations of the equation
\section{Global existence for some perturbations of non decaying smooth  solutions}
{
We are looking for  solutions of the $KP$ equation of the form $ V= \psi +v $ with $ \psi $ that satisfies 
$$
(\psi_t + \psi_{3x} + \psi \psi_x)_x +\epsilon \psi_{yy}=g_x \;\text{with } g \in L^\infty_T H^{s,0}
$$
Then $ v$ is solution to 
\begin{equation}\label{eqv}
\Bigl(v_t + v_{3x} + v v_x + (\psi v)_x + g\Bigr)_x +\epsilon v_{yy}=0 
\end{equation}
Therefore $ v $ satisfies the same equation as $ u$ with two supplementary terms $ (\psi v)_x $ and $ g $ that have to be taken into
 account. 
 \subsection{Refined Strichartz estimates and trilinear estimates for \eqref{eqv}}
 We notice that $ v$ satisfies \eqref{eq:kpI} with $ \partial_x(v^2/2)$ replaced by $ (\frac{1}{2} v^2 + \psi v)_x + g $. We thus have to estimate the contribution of $( \psi v)_x + g $ in the refined Strichartz estimates and in the trilinear estimate. We will always put the $ L^2_{xy}$-norm on these two terms.
 We summarize below the changes we have to take into account in the different estimates. \\
$\bullet $ {\it Changes in the Refined Strichartz estimates of Section 3} \\
For the Strichartz estimate we use the  $ L^1_T L^2_{xy} $-norm.  In \eqref{tyty} for the estimate on $ \|P_N v\|_{L^4_T L^4_{xy}}^4 $ this adds the following terms
\begin{align*}
 \sum_{j\in J_N}  &\Bigl( ( TN^{-2})^\frac{3}{4} \Bigl(N\|P_N (\psi v)  \|_{L^4_{I_j}  L^2_{xy}} + \|P_N g  \|_{L^4_{I_j}  L^2_{xy}} \Bigr)\Bigr)^{4} 
 \nonumber\\
& \lesssim  T^3 \Bigl( N^{-2} \|P_N (\psi v)  \|_{L^4_{T}  L^2_{xy}}^4  + N^{-6} \|P_N g  \|_{L^4_{T}  L^2_{xy}}^4 \Bigr)
\end{align*}
that leads, using \eqref{prod3},  for $ s>3/4 $  (in a non optimal way) to 
\EQ{
\|  J_x^{s-(\frac{1}{2}+)} v \|_{L^{4}_T L^4_{xy}}
\lesssim&   \| v \|_{L^\infty_T H^{s,0}} +
 \|J_x^{s-} v^2 \|_{L^{4}_T L^1_{xy} }   \\
&+ T^\frac{3}{4} \Bigl(  \|J_x^{s-\frac{3}{4}} \psi\|_{L^4_T L^\infty_{xy}} \|v \|_{L^\infty_{T}  H^{s-\frac{3}{4}, 0}}+ \|g\|_{L^\infty_T H^{s-2}}\Bigr). \label{ImprovedStri1new}
}
In the same way we have to add the following contribution in the estimate \eqref{toti1} 
\begin{align*}
 \sum_{j\in J_N}  &\Bigl( ( TN^{-\frac{4}{3}})^{\frac{1}{2}+} (N\|P_N (\psi v)  \|_{L^{2+}_{I_j}  L^2_{xy}} + \|P_N g  \|_{L^{2+}_{I_j}  L^2_{xy}} )\Bigr)^{2+} 
 \nonumber\\
& \lesssim  T^{\frac{1}{2}+} \Bigl( N^{\frac{1}{3}-} \|P_N (\psi v)  \|_{L^{2+}_{T}  L^2_{xy}} + N^{-(\frac{2}{3}+)} \|P_N g  \|_{L^{2+}_{T}  L^2_{xy}} \Bigr)^{2+}
\end{align*}
that leads for $s>3/4 $  to 
\EQ{
\|  J_x^{s-(\frac{2}{3}+)} v \|_{L^{2+}_T L^\infty_{xy}}
\lesssim &  \| v \|_{L^\infty_T H^{s,0}} +\| J_x^s  v^2 \|_{L^{2+}_T L^\frac{4}{3}_{xy} } \\
&+ T^\frac{1}{2}\Bigl(  \|J_x^{s-\frac{1}{3}} \psi\|_{L^{2+}_TL^\infty_{xy}} \|v \|_{L^\infty_{T}  H^{s-\frac{1}{3}, 0}}+ \|g\|_{L^\infty_T H^{s-\frac{4}{3}}}\Bigr).   \label{ImprovedStri2new}
}
Finally since in \eqref{estcourt} we already used  the $ L^1_t L^2_{xy} $-norm of $f$ it is direct to check that \eqref{ImprovedStri5} becomes
\EQ{
\|  P_{\lesssim N} v \|_{L^{2}_I L^\infty_{xy}}
 \lesssim & T^{\frac{1}{4}-} N^{-\frac{1}{4}+}  \Bigl( \|J_x^{\frac{3}{4}} v \|_{L^\infty_I L^2_{xy}} +
  \|J_x^{\frac{3}{4}}  v^2 \|_{L^{2}_I L^2_{xy} }\\
  &\qquad + \|J_x^\frac{3}{4} \psi\|_{L^2_T L^\infty_{xy}}\|v\|_{L^\infty_T H^{\frac{3}{4},0}}+ \| g\|_{L^\infty_T H^{-\frac{1}{4},0}}  \Bigr). \label{ImprovedStri5new}
}
$\bullet $ {\it Changes in the trilinear estimates of Section 4} \\
In the trilinear estimate \eqref{tritri} for the high frequencies $ N_1 $ and $N_2$  the supplementary term to add is given by 
$$
T^{1/2}  \Bigl(  \|P_{N_i}(\psi v)\|_{L^2_T L^2_{xy}}+ N_i^{-1} \|P_{N_i} g\|_{L^\infty_T L^2_{xy}} \Bigr) 
$$
that leads, using \eqref{prod1}, to the following supplementary terms in the first parenthesis in the right-hand side of \eqref{desdes} for the a priori estimate on $ v$ :
\begin{equation}\label{95}
T  \Bigl(  \|J^{s+}_x \psi\|_{L^\infty_{Txy}}^2 \|v \|_{L^\infty_T H^{s,0}_\omega}^2+ \|g \|_{L^\infty_T H^{s-1,0}_\omega}^2 \Bigr) \; .
\end{equation}
For the low frequency $ N_3$  the supplementary term is given by 
\begin{align*}
&\Bigl(\frac{N_3}{N_1} \Bigr)^{1/2} \cro{N_3}^{0+} \Bigl( N_3 \|P_{N_3} (\psi v) \|_{L^2_I L^2_{xy}} + \|P_{N_3} g\|_{L^2_I L^2_{xy}}\Bigr) \\
& \lesssim \Bigl(\frac{N_3}{N_1} \Bigr)^{1/2}\cro{N_3}^{0+}(T N^{-1})^\frac{1}{2}  \Bigl( N_3 \|P_{N_3} (\psi v) \|_{L^\infty_T L^2_{xy}} + \|P_{N_3}g\|_{L^\infty_T L^2_{xy}}\Bigr)
\end{align*} 
that  leads to the following supplementary terms in  the second parenthesis in the right-hand side of \eqref{desdes} for the a priori estimate on $ v$ :
\begin{equation}\label{96}
T^\frac{1}{2}  \Bigl(  \|J^{\frac{1}{2}+}_x \psi\|_{L^\infty_{Txy}} \|v \|_{L^\infty_T H^{\frac{1}{2}+,0}}+ \|g\|_{L^\infty_T H^{-\frac{1}{2}+,0}}\Bigr) 
\end{equation}
Finally we notice that  the difference $ w $ of two solutions $ v_1 $ and $ v_2 $  to \eqref{eqv} satisfies 
\begin{equation}\label{eqw2}
\Bigl(w_t + w_{3x} + (zw+\psi w)_x \Bigr)_x +\epsilon w_{yy}=0 
\end{equation}
where $z=v_1+v_2$.
In the trilinear \eqref{tritri2} $ \psi w$ will lead exactly in the same way to the supplementary terms
$T^{1/2}    \|P_{N_i}(\psi w)\|_{L^2_T L^2_{xy}}$ for the high frequencies $ N_1, \, N_2 $ and 
\EQN{
\Bigl(\frac{N_3}{N_1} \Bigr)^{1/2}\cro{N_3}^{0+}(T N^{-1})^\frac{1}{2}   N_3 \|P_{N_3} (\psi w) \|_{L^\infty_T L^2_{xy}} 
}
for the low frequency $N_3$. 
 \subsection{A priori estimates for \eqref{eqv}} We now summarize the changes on the a priori estimates on the solution $ v$ and on the difference of two solutions $ w=v_1-v_2$ to \eqref{eqv}. \vspace*{2mm} \\
$\bullet $ {\it Changes in the $ H^{s,0} $-estimate on $ v $ } \\
For $ N\lesssim 1 $ it directly holds 
$$
\sum_{0<N\lesssim 1} \cro{N}^{2s}\Bigl|\int_0^t \int_{\R^2} P_{N} ((\psi v)_x +g)  P_{N} v \Bigr| \lesssim (\|\psi\|_{L^\infty_{xy}} \|v\|_{L^\infty_T L^2_{xy}}
+ \|g\|_{L^\infty_T L^2_{xy}}) \|v\|_{L^\infty_T L^2_{xy}}.
$$
For $ N\gg 1 $, by commutator estimates we have 
\EQN{
&\sum_{N\gg 1} \om_N^2\cro{N}^{2s} \Bigl|\int_0^t \int_{\R^2} P_{N} ((\psi v)_x+g)  P_{N} v \Bigr|\\
\lesssim& \sum_{N\gg 1} \int_0^t \|P_N g\|_{H^{s,0}_\om} \|P_N v\|_{H^{s,0}_\om}
\\
& + \sum_{N\gg 1}   \int_0^t (\|P_{\ll N} \psi_x\|_{L^\infty_{xy}}  \|P_N v\|_{H^{s,0}_\om}^2+\om_N  N^{s+1}  \|P_{\gtrsim N} \psi\|_{L^\infty_{xy}} \|P_N v\|_{H^{s,0}_\om}
 \|v\|_{L^2_{xy}})\\
& \lesssim T \|J_x^{s+1+} \psi \|_{L^\infty_{Txy}} \|v\|_{L^\infty_T H^{s,0}_\om}^2  +T \|g\|_{L^\infty_T H^{s,0}_\om} \|v\|_{L^\infty_T H^{s,0}_\om}.
}
Combining these last estimates with \eqref{95}-\eqref{96} we eventually get the following $ H^{s,0}_\om$-estimate on $ v$ 
\EQ{\label{estimv}
&\|v\|_{L^\infty_T H^{s,0}_\om}^2\\
 \lesssim & \|v_0\|_{H^{s,0}_\om}^2+ T^{1/2} \|v \|_{L^2_T L^\infty_{xy}} \| v\|_{L^\infty_T L^2_{xy}}^2 +T \|J_x^{s+1+} \psi \|_{L^\infty_{Txy}} \|v\|_{L^\infty_T H^{s,0}_\om}^2 \\
& +T  \|g_x\|_{L^\infty_T H^{s,0}_\om} \|v\|_{L^\infty_T H^{s,0}_\om}+(\| v\|_{L_T^\infty H^{s,0}_\om}^2+\|g\|^2_{H^{s,0}} )(\|v\|_{L^\infty_T H^{\frac{3}{4}+,0}}+
\|g\|_{L^\infty_T L^2_{xy}})\\
&\times(1+ \|v\|_{L^2_T L^\infty_{xy}}^2+ \|J_x^{s+} \psi \|_{L^\infty_{Txy}}^2)(1+ \|v\|_{L^2_T L^\infty_{xy}}^2+\|J_x^{\frac{1}{2}+} \psi \|_{L^\infty_{Txy}}+\|v\|_{L^\infty_T H^{\frac{3}{4}+,0}}^2).
}
$\bullet$ {\it Changes in the $ H^{s-1,0} $-estimate on $ w=v_1-v_2 $ } We proceed exactly in the same way.
For the estimate on the low frequency part of $ w$  we use the Duhamel formulation and apply \eqref{non-hom} with $(q_1,r_1)=(q_2,r_2)=(\infty,2) $.  It is then straightforward to check that \eqref{lowf} can be then replaced by 
\begin{equation}
 \|P_{\lesssim 1} w\|_{L^\infty_T \overline{H}^{s-1,0}}   \lesssim  \| w(0)\|_{\overline{H}^{s-1,0}}+
  T^{1/2} (\|z\|_{L^\infty_T H^{\frac{1}{4}+,0}}+\|J_x^{\frac{1}{4}+}\psi\|_{L^\infty_{Txy}} )\|w\|_{L^\infty_T H^{-\frac{1}{4},0}} \label{lowf2} \; .
 \end{equation}
For $ N\gg 1 $, by commutator estimates we have 
\begin{align*}
&\sum_{N\gg 1} \cro{N}^{2(s-1)} \Bigl|\int_0^t \int_{\R^2} P_{N} (\psi w) \partial_x P_{N} w \Bigr|\\
&\lesssim  \sum_{N\gg 1}   \int_0^t (\|P_{\ll N} \psi_x\|_{L^\infty_{xy}} \|P_N w\|_{H^{s-1,0}}+ N^s  \|J_x^{\frac{3}{4}+} P_{\gtrsim N} \psi\|_{L^\infty_{xy}}
\|P_N w\|_{H^{-\frac{1}{4},0}}) \|P_N w\|_{H^{s-1,0}}\\
& \lesssim T \|J_x^{s+\frac{3}{4}+} \psi \|_{L^\infty_{Txy}} \|w\|_{L^\infty_T H^{s-1,0}}^2 \; .
\end{align*}
This eventually leads to  the following $ \overline{H}^{s-1,0}$-estimate on $ w$ 
\EQN{
 &\|w\|_{L^\infty_T  \overline{H}^{s-1,0}_{\om}}^2\\
 \lesssim &\|w(0)\|_{ \overline{H}^{s-1,0}_{\om}}^2 +
 (1+\|J_x^{\frac{1}{4}+} \psi \|_{L^\infty_{Txy}})\|J_x^{s+\frac{3}{4}+} \psi \|_{L^\infty_{Txy}} \|w\|_{L^\infty_T H^{s-1,0}}^2\\
 &+\Bigl[\Bigl(\sum_{i=1}^2(1+  \|v_i\|_{L^\infty_T H^{\frac{3}{4}+,0}_{xy}})(1+
  \|v_i\|_{L^2_T L^\infty_{xy}}+\|J_x^{s+} \psi\|_{L^\infty_{Txy}}) ( \|v_i\|_{ H^{s,0}_{\om}}+\|g\|_{L^\infty_T H^{s,0})}
  \Bigr)
  \Bigr]\\
 & +\|w\|_{L^\infty_T\overline{H}^{s-1,0}_{\om}}^2 \times \Bigl[ 1+ \|J_x^{s+1} \psi \|_{L^\infty_{Txy}}^2+\sum_{i=1}^2 \Bigl( (1+\|v_i\|_{L^2_T L^\infty_{xy}})^2 \|v_i\|_{L^\infty_T H^{\frac{3}{4}+,0}}^2 \\
 &\qquad \qquad \qquad  \qquad \quad +\| J_x^{\max(0,s-1)+} v_i \|_{L^{2}_T L^\infty_{xy}}^2 + \|D_x^{s-(\frac{1}{2}+)} v_i\|_{L^4_{T}L^4_{xy}}^2 \Bigr)\Bigr] .
}
\subsection{LWP in $ H^{s,0} $ for \eqref{eqv}}

 It follows the same lines as in Section \ref{sect7}. However, we have to rely on a LWP result for smooth solutions to \eqref{eqv}. We  are not aware of such LWP result in the literature but it is well-known that (following for instance \cite{MoSaTz4}) such LWP result can be proven by a standart compactness method in $ H^s(\R^2) $ for $s>2 $ assuming for instance that $ \psi \in L^\infty_T W^{\infty,\infty}(\R^2) $ and 
  $g\in L^\infty_T H^\infty(\R^2)$. Note that these smoothness requirements on $ \psi $ and $ g $ are not problems since by truncating the space frequencies of $ \psi$ above some $N\ge 1$, Hypothesis \ref{hyp1} ensures that they are satisfied for any fix $ N\ge 1$. We may then pass to the limit as $ N\to +\infty$. One other difference is that now  $v\in L^\infty(0,T;H^{s,0})$ is a solution to \eqref{eqv}  with $(\psi, g) $ on $ ]0,T[ $ if and only if, for some $ 0<\lambda\le 1$,  $(t,x,y)\mapsto u_\lambda(t,x,y):=\lambda^{2} u(\lambda^{3}t, \lambda x, \lambda^2 y)\in L^\infty(0, \lambda^{-3} T;H^s)  $ is  a  solution to \eqref{eqv}  on $ ]0,\lambda^{3} T[ $ with $\psi_\lambda(t,x,y)= \lambda^2 \psi(\lambda^3 t,\lambda x,\lambda^2 y) $ and $ g_\lambda(t,x,y)=\lambda^5 g (\lambda^3 t,\lambda x,\lambda^2 y)$.
 We notice that for $ \theta\ge 0 $, 
\begin{equation}\label{rescale2}
\|J_x^{\theta}\psi_{\lambda}(t)\|_{L^\infty_{xy}} \lesssim\  \lambda^{2}
 \|J_x^{\theta} \psi(\lambda^3 t)\|_{L^\infty_{xy}}
\quad \text{and} \quad \|g_\lambda\|_{H^{\theta,0}} \lesssim\  \lambda^{7/2} \|g(\lambda^3 t)\|_{H^{\theta,0}} , \quad \forall 0\le t\le  \lambda^{-3} T \; .\end{equation}
Since we have again positive power of $ \lambda $ the dilation argument still holds as well for the perturbed equation \eqref{eqv} that leads to the LWP result for $ v$ in $ H^{s,0}$, $ s>3/4 $, for $ \psi  $  satisfying Hypothesis \ref{hyp1}. Note that 
 since $ g\in C(\R;H^{s,0})$, on any compact interval $ [0,T] $ we can find a frequency envelop 
 $\{\omega_N\}$ adapted to $ g$ uniformly on  $[0,T]$  that means $ \sup_{t\in [0,T]} \|g\|_{H^{s,0}_\omega} <\infty $. 
\subsection{GWP for a perturbation in $E^s(\R^2)$ in the KP-I case}
As above  we summarize below the contributions of the additional terms in this configuration. \\
$\bullet $ {\it  Changes in the $H^{s-1,0} $-estimate on $ \partial_x^{-1} v_y $ }. \\
We notice that $ \vartheta=\partial_x^{-1} v_y $ satisfies this time 
$$
(\vartheta_t + \vartheta_{xxx} +u \vartheta_x + \psi_y v + \psi \vartheta_x + \partial_x^{-1} g_y )_x -  \vartheta_{yy} = 0 \; .
$$
For $ N\lesssim 1 $ it directly holds 
\EQN{
&\sum_{0<N\lesssim 1} \cro{N}^{2(s-1)}\Bigl|\int_0^t \int_{\R^2} P_{N} \Bigl(\psi_y  v+ \psi \partial_x (\partial_x^{-1} v_y) + \partial_x^{-1} g_y\Bigr)  \partial_x^{-1} P_{N} v_y \Bigr| \\
 \lesssim &(\|\psi_y \|_{L^\infty_{xy}} + \|J_x^{1+} \psi \|_{L^\infty_{xy}}) ( \|v\|_{L^\infty_T L^2_{xy}}^2+ \|\partial_x^{-1} v_y\|_{L^\infty_T L^2_{xy}}^2)\\
 &+ \|\partial_x^{-1} g_y\|_{L^\infty_T L^2_{xy}} \|\partial_x^{-1} v_y\|_{L^\infty_T L^2_{xy}}
}
For $ N\gg 1 $, by commutator estimates we have 
\EQN{
&\sum_{N\gg 1} \cro{N}^{2(s-1)} \Bigl|\int_0^t \int_{\R^2} P_{N}  \Bigl(\psi_y  v+ \psi \partial_x (\partial_x^{-1} v_y) + \partial_x^{-1} g_y\Bigr)  \partial_x^{-1} P_{N} v_y \Bigr| \\
\lesssim &(\|J_x^{(s-1)+}\psi_y \|_{L^\infty_{xy}} + \|J_x^{(s-1)+} \psi_x \|_{L^\infty_{xy}}) ( \|v\|_{L^\infty_T H^{s,0}}^2+ \|\partial_x^{-1} v_y\|_{L^\infty_T H^{s,0}}^2)\\
&+ \|\partial_x^{-1} g_y\|_{L^\infty_T H^{s,0}} \|\partial_x^{-1} v_y\|_{L^\infty_T H^{s,0}} \; .
}
 This shows that, under Hypothesis \ref{hyp2}, we may obtain an acceptable a priori estimate on $\|\partial_x^{-1} v_y \|_{H^{s-1,0}}^2$ as in Proposition \ref{pro8}.\\
$\bullet $  {\it Changes in the estimate on the Energy} \\
 Since we 	are now working with a finite energy solution $ v$ , \eqref{eqv} may be replaced by its ``integrated" version 
 \begin{equation} \label{intequ}
 v_t + v_{3x} + v v_x +\frac{1}{2} (\psi v)_x + g- \partial_x^{-1} v_{yy}=0 \; .
\end{equation}
 First taking the $ L^2$-scalar product of \eqref{intequ} with $ v$ we get
  \begin{equation}
 \frac{d}{dt} \|v\|_{L^2}^2\lesssim \|\psi_x \|_{L^\infty_{xy}} \|v\|_{L^2}^2 + \|g\|_{L^2_{xy}} \|v\|_{L^2} \; .
 \end{equation}
 that proves that 
 \begin{equation} \label{estL2}
 \|v(t)\|_{L^2_{xy}}^2 \lesssim (\|v_0\|_{L^2_{xy}}^2+T\|g\|_{L^\infty_T L^2_{xy}}) \exp \Bigl( (\|\psi_x \|_{L^\infty_T L^\infty_{xy}}+\|g\|_{L^\infty_T L^2_{xy}}) \, t  \Bigr) 
 \end{equation}
 Then we estimate the time derivative of the energy. Formally we take the  $ L^2$-scalar product of \eqref{intequ} with $ v_{xx} + \partial_x^{-2} v_{yy} + v^2$.
  To justify rigorously the calculus we may proceed as in \cite{MoSaTz4} by performing an exterior regularization of
\eqref{intequ} by a sequence of smooth functions $\varphi$ that
cut the low and the high $x$- frequencies. We do not want to enter into these details here and proceed as all calculus were fully  justified for our solutions.

  Using that the energy is conserved for smooth solutions of the KP-I  equation, the only contributions in the derivative of the energy of $ v$ will come from the supplementary terms $ (\psi v)_x $ and $ g$. We thus have only to control the contributions $ A$ and $ B $ of respectively $ g$ and $ \partial_x( \psi v ) $. \\
Obviously,  it holds 
\begin{align*} 
|A| &  \lesssim (\|g_x\|_{L^2_{xy}}+ \|\partial_x^{-1}g_y\|_{L^2_{xy}}) \|v\|_{ L^2_{xy}}+\|g\|_{L^2_{xy}}\|v\|_{L^4_{xy}}^2\; .
\end{align*}
whereas by integration by parts  
\begin{align*} 
|B| & \lesssim (\|\psi_{xx} \|_{L^\infty_{xy}}+ \|\psi_{x} \|_{L^\infty_{xy}}+\|\psi_{y} \|_{L^\infty_{xy}})\|v\|_{E^1}^2 +  \|\psi_x\|_{L^\infty_{xy}} \|v\|_{L^3_{xy}}^3 \\
\end{align*}
with, in view of  \eqref{malak}, $ \|v\|_{L^4_{xy}}^2\lesssim \|v\|_{E^1} $ and 
$$
 \|v\|_{L^3_{xy}}^3\lesssim \|v\|_{L^2_{xy}}^{3/2}  \|v_x\|_{L^2_{xy}}  \|\partial_x^{-1} v_y\|_{L^2_{xy}}^{1/2} \lesssim
  \max\Bigl(  \|v\|_{L^2_{xy}} \|v\|_{E^1}^2, \frac{1}{2} \|v\|_{E^1}^2+ C \|v\|_{L^2_{xy}}^6 \Bigr) \; .
$$
Therefore $ \|v\|_{E^1} $ satisfies the following integral inequality 
\EQN{
\|v(t)\|_{E^1}^2 \lesssim& E(v_0) +\|v(t)\|_{L^2_{xy}}^6+\int_0^t \|g(\tau)\|_{E^1} \, d\tau\\
&+ \int_0^t \Bigl(\|g(\tau)\|_{E^1}+\| \psi_{2x}(\tau) \|_{L^\infty_{xy}}+\| \psi_x(\tau) \|_{L^\infty_{xy}}+ \| \psi_y(\tau) \|_{L^\infty_{xy}}\\
&\qquad +\|\psi_x(\tau)\|_{L^\infty_{xy}}\|v(\tau)\|_{L^2_{xy}}
\Bigr) \|v(\tau)\|_{E^1}^2\, d\tau
}
that proves that $t\mapsto  \|v(t)\|_{E^1}$ 
 is bounded on bounded intervals by Gronwall Lemma and lead to the global existence result. Finally it remains to prove the continuity of  $ t\mapsto  \partial_x^{-1} v_y $ with values in $ H^{s-1,0}(\R^2) $ as well as the continuity of
 the map $ v_0 \mapsto \partial_x^{-1} v_y $ from $E^s(\R^2) $ into $ C([0,T];H^{s-1,0}(\R^2)) $ for $ T>0 $. Actually this follows from the same type of arguments as in Section \ref{sect8} but by replacing the conservation of the energy by the fact that $t\mapsto E(v(t)) $ is continuous (see \cite{MoSaTz4}, Section 7] for similar considerations).

\section*{Acknowlegdements}
The authors would like to thank J.-C. Saut for having informed them  about  some exact solutions to the KP-II equation. They are grateful to the anonymous Referees for valuable remarks and suggestions that improved this manuscript. 
\section*{Funding Declaration}
Z. Guo is supported by ARC FT230100588 from Australian Research Council.

\end{document}